\documentclass[reqno, 11pt]{amsart}

\usepackage{amsfonts, amsthm, amsmath, amssymb}
\usepackage[pagebackref=true]{hyperref}

\usepackage[margin=3.5cm]{geometry}

\usepackage{tikz-cd}
\usepackage{tikz}
\usepackage{pgfplots}
\usetikzlibrary{datavisualization}

\makeatletter
\@namedef{subjclassname@2020}{%
  \textup{2020} Mathematics Subject Classification}
\makeatother

\RequirePackage{mathrsfs} \let\mathcal\mathscr
\numberwithin{equation}{section}

\renewcommand{\geq}{\geqslant}
\renewcommand{\leq}{\leqslant}

\newtheorem{theorem}{Theorem}[section]
\newtheorem{corollary}[theorem]{Corollary}
\newtheorem{lemma}[theorem]{Lemma}
\newtheorem{proposition}[theorem]{Proposition}

\newtheorem{conjecture}[theorem]{Conjecture}

\theoremstyle{definition}
\newtheorem{definition}[theorem]{Definition}
\newtheorem{remark}[theorem]{Remark}

\newcommand{\Q}{\mathbb{Q}}
\newcommand{\NN}{\mathbb{N}}

\newcommand{\R}{\mathbb{R}}
\newcommand{\OO}{\mathcal{O}}

\newcommand{\F}{\mathbb{F}}

\newcommand{\p}{\mathcal{P}}
\newcommand{\N}{\mathbb{N}}
\newcommand{\q}{\mathfrak{q}}
\newcommand{\eps}{\varepsilon}
\newcommand{\abs}[1]{\left|#1\right|}

\def\End{\operatorname{End}}

\def\li{\operatorname{li}}

\def\QQ{\mathbb{Q}}
\def\ZZ{\mathbb{Z}}
\def\FF{\mathbb{F}}
\def\NN{\mathbb{N}}

\def\ve{\varepsilon}

\begin{document}
\date{\today}	
	
\title{Strong divisibility sequences and sieve methods}

\author{Tim Browning}
	\address{IST Austria\\
		Am Campus 1\\
		3400 Klosterneuburg\\
		Austria}

\email{tdb@ist.ac.at\\
matteo.verzobio@gmail.com}	
\author{Matteo Verzobio}


\begin{abstract}
We investigate strong divisibility sequences and produce lower and upper bounds for the density of integers in the sequence
which only have (somewhat) large prime factors. We focus on the special cases of Fibonacci numbers and elliptic divisibility sequences, discussing the limitations of our methods. At the end of the paper there is an appendix by Sandro Bettin on divisor closed sets, that we use to study the density of prime terms that appear in strong divisibility sequences.
\end{abstract}

	\maketitle
	\thispagestyle{empty}
\setcounter{tocdepth}{1}
\tableofcontents

	\section{Introduction}

The arithmetic of sparse integer sequences continues to present a considerable challenge in analytic number theory. In this paper we shall use elementary sieve theory to prove some crude
arithmetical properties of sparse sequences that admit a suitable divisibility structure.  We begin by discussing some well-known sparse sequences and related literature.

\subsubsection*{Shifted exponentials}

Given $a>1$ and non-zero $b\in \ZZ$, 
a notorious open problem is to determine whether the sequence $a^n-b$ takes infinitely many prime values as $n$ runs over $\NN$. These sequences are sparse, having only $O(\log N)$ elements in the interval $[1,N]$.   Recent work of Grantham and Granville \cite{granville23} speculates on this for the sequence $a\cdot 2^n-b$ when $\gcd(a,b)=1$ and $a>0$; they conjecture that either there are finitely many primes in the sequence, or else the number of primes in the interval $[1,N]$ grows like $c_{a,b}\log N$ for a suitable constant $c_{a,b}>0$.
(When $a=b=1$, which is the case of Mersenne primes, their heuristic 
suggests the asymptotic formula should hold with $c_{1,1}=e^\gamma /\log 2$.)
In the opposite direction, 
in his book on sieve methods \cite{hooley}, Hooley conjectured that almost all numbers in the sequence $2^n+5$ are  composite.
Hooley's conjecture has now been resolved by 
J\"arviniemi and Ter\"av\"ainen \cite{joni}, subject to GRH and a form of the pair correlation conjecture.

\subsubsection*{Markoff numbers}

The sequence of Markoff numbers is obtained by listing 
all coordinates that arise as positive integer solutions to the Diophantine equation
$x_1^2+x_2^2+x_3^2=3x_1x_2x_3$. Thanks to work of 
Bourgain, Gamburd and Sarnak \cite{BGS}, we know that almost all Markoff numbers are composite. 
This is a sparse sequence, since the interval $[1,N]$ contains $O((\log N)^2)$ 
elements of  the sequence, by work of Zagier \cite{zagier}. 

\subsubsection*{Elliptic curves}
Let $E$ be an elliptic curve defined by a Weierstrass equation with integer coefficients
and let 
 $P\in E(\Q)$ be a non-torsion point. 
For each $n\in \NN$ we may define
$nP=(x_n:y_n:z_n)$, for relatively coprime $x_n,y_n,z_n\in \ZZ$.
 As we shall see in Section~\ref{sec:eds}, we must have $x_n=a_nb_n$ and $z_n=b_n^3$, for $a_n,b_n\in \ZZ$. The sequence $\{b_n\}_{n\in \N}$
is even sparser than the previous examples, containing only $O(\sqrt{\log N})$ integers from the interval $[1,N]$. 
Nonetheless, it is still possible to prove non-trivial results about their arithmetic. 
 In \cite[Section~4.4]{kowalski2008large}, Kowalski uses the large sieve to 
 show that elements of the sequence with few prime divisors have density $0$. More recently, similar sequences have been studied by 
Bhakta, Loughran, Rydin Myerson and Nakahara 
\cite{brauersieve} in the context of local solubility for families of conics parametrised
by elliptic curves. In 
\cite[Thm.~1.1]{brauersieve}, for example, they show that the set of $n\in \NN$ for which $y_n$ is a sum of two squares has density $0$, if $P$ belongs to the connected component of the identity of $E(\R)$. Due to the fact that the sequence is so sparse, it 
has been  conjectured  by Einsiedler, Everest and Ward \cite{primality} that 
 there are only finitely many prime numbers appearing in the sequence $b_n$.

\medskip

Our results in this paper pertain to any {\em strong divisibility sequence} (SDS), which is a sequence of positive integers $\{x_n\}_{n\in \NN}$ such that 
	$$
	\gcd(x_n,x_m)=x_{\gcd(n,m)}.
	$$
for all $m,n\geq 1$. We may henceforth assume without loss of generality that  $x_1=1$, on noting that the general case can be deduced from this case by applying the results to the sequence $x_n/x_1$.
The most basic example of a SDS is the 
sequence $x_n=n$, for which our results are all well-known or trivial. Further examples are provided by {\em elliptic divisibility sequences}, 
corresponding to the sequence $\{b_n\}_{n\in \NN}$
discussed above, and 
{\em Lucas sequences of the first kind}. The latter include 
sequences of 
{\em Fibonacci numbers} 
$$
F_n=\frac{\phi^n-(-\phi)^{-n}}{\sqrt{5}},
$$ 
where $\phi$ is the golden ratio, 
as well as {\em Mersenne numbers} $2^n-1$. 

There is a rich literature around the arithmetic of 
such numbers; for example, 
in 1965 Erd\H{o}s \cite{paul} conjectured that $\frac{1}{n}P(2^n-1)\to \infty$ as $n\to \infty$, where 
$P(2^n-1)$
is the size of the largest prime factor of $2^n-1$.
Stewart \cite{cam} has verified this conjecture in the wider context of  Lucas sequences, showing in particular that $P(2^n-1)>n \exp(c \log n/\log\log n)$, for a suitable constant $c>0$.
In a different direction, 
Luca and St{\u{a}}nic{\u{a}} \cite{luca2005prime} provide heuristics about the size of 
$\omega(x_n)$ for the Lucas sequences $x_n=(a^n-1)/(a-1)$, for integer $a>1$, conjecturing that 
$\omega(x_n)\geq (1+o(1))\log n \log\log n$ for almost all $n$.
Recent work of Kontorovich and Lagarias \cite{alex} posits a new {\em toral affine sieve conjecture}, 
under which it is possible to study the total number of prime divisors of products like $F_nF_{n+1}$. Thus it follows from \cite[Theorem~1.5]{alex} that 
$$
\Omega(F_n)+\Omega(F_{n+1})\gg \log n
$$ 
for all sufficiently large $n$, conditionally on their toral affine sieve  conjecture. 
Finally, a classification of those SDS which are simultaneously a {\em linear recurrence sequence} has been  completed   recently by Granville \cite{granville}.

\begin{definition}\label{def:md}
Let $\{x_n\}_{n\in \NN}$ be a SDS and 
let $d\in \NN$. Define $m_d$ to be the smallest positive integer such that $d\mid x_{m_d}$. If no such integer exists, we put $m_d=\infty$.
\end{definition}

There are two major challenges inherent in trying to apply sieve theory to a SDS. 
We shall see in Remark \ref{rem:LOD} that
\begin{equation}\label{eq:LOD}
\#\{n\leq N: d\mid x_n\} =\frac{N}{m_d}+O(1), 
\end{equation}
so that we have an associated {\em density function} $g(d)=1/m_d$. 
Unfortunately, in most cases of interest the function $g(d)$ is not multiplicative, whereas 
one of the basic sieve axioms is that $g(d)$ should be a non-negative multiplicative function, with 
$0<g(p)\leq 1$ for any prime $p$.
The second issue to contend with is the fact that a SDS can grow exponentially, as in the case of Lucas sequences, or even 
quadratic exponentially, as in the case of elliptic divisibility sequences. Most successful applications of sieve theory involve sequences with at most polynomial growth.

	The following is one of our main results and is based on the most basic sieve of 
	Eratosthenes--Legendre.
	
	\begin{theorem}\label{thm:gen}
	Let $\{x_n\}_{n\in \NN}$ be a SDS with the property that there 
		exists 
		 $\alpha>0$ such that, for all but finitely many primes $p$, we have 
			$
			m_p<p^\alpha.
			$
				Then
		$$
		\#\left\{n\leq N :
		p\mid x_n \implies p> \tfrac{1}{2}(\log N)(\log\log N) \right\}
		\gg \frac{N}{\log\log N}.
		$$
	\end{theorem}
	
	This result provides a lower bound for the number of $z$-rough numbers in a SDS 
	$\{x_n\}_{n\in \NN}$, with $n\leq N$ and $z=\tfrac{1}{2}(\log N)(\log\log N)$.
		We will see that the hypothesis of the theorem is satisfied by Lucas sequences of the first kind, and by elliptic divisibility sequences. 
Moreover, note that if $\{x_n\}_{n\in \NN}$ and $\{y_n\}_{n\in \NN}$ are SDS  that satisfy the hypothesis of Theorem~\ref{thm:gen}, then also the sequence $\{\gcd(x_n,y_n)\}_{n\in \NN}$ satisfies the hypothesis.

We shall  illustrate our most general results with the sequence  $\{F_n\}_{n\in \NN}$  of Fibonacci  numbers. Note that $\Omega(F_n)=O(n)$, since 
$\Omega(m)=O(\log m)$ for any $m\in \NN$. (In fact, the latter bound is optimal, as one sees by considering the case in which $m$ is a prime power.) 
The following shows that we can often improve on this trivial bound for $\Omega(F_n)$.

\begin{corollary}
We have $$\#\{n\leq N: \Omega(F_n)\leq \frac{n}{\log \log n}\} \gg \frac{N}{\log\log N}.$$
\end{corollary}

\begin{proof}
This follows from Theorem~\ref{thm:gen}, on noting that 
 $\Omega(F_n)\leq \frac{n}{\log\log n}$ if $p>C\log N$ for all $p\mid F_n$, for a suitable constant $C>0$.
\end{proof}

Under additional hypotheses, we can also prove an upper bound for the kind of set considered in 
Theorem \ref{thm:gen}. We say that the term $x_n$ has a {\em primitive prime divisor} if there exists a prime $q$ such that $q\mid x_n$, but $q\nmid x_k$ for $k<n$.

\begin{theorem}\label{thm:upperf}
Let $\{x_n\}_{n\in \NN}$ be a SDS with the property that 
				$x_n$ has a primitive prime divisor 
				for all but finitely many terms.
Then, there exists a strictly increasing continuous function $f:\R_{\geq 1}\to \R$ such that $n\leq f(z)$ implies $x_{n}\leq z$,  for all $z\geq 1$. Moreover, for any such function $f$, we have 
	$$
	\#\left\{n\leq N :
	p\mid x_n \implies  p> z \right\}\ll \frac{N}{\log f(z)}+f(z)^2,
	$$
	for any $z$ such that $f(z)>1$.
\end{theorem}

According to \cite[Theorem 1]{schinzel} and 
\cite[Proposition 10]{silverman}, respectively, 
both Lucas sequences of the first kind (except for some trivial examples) and elliptic divisibility sequences satisfy the hypothesis of Theorem \ref{thm:upperf}.

Our next result is concerned with the density of prime terms that appear in a SDS $\{x_n\}_{n\in \NN}$.   
It is well-known that a Mersenne number $2^n-1$ is composite if $n$ is composite,
so that the set of Mersenne primes among the sequence $\{2^n-1\}_{n\in \NN}$ has 
density $0$.
We shall generalise this result to a more general SDS, as follows. 

		\begin{theorem}\label{thm:cardprimes}
				Let $\{x_n\}_{n\in \NN}$ be a SDS and let
				$
				\mathcal{A}=\{n\in \N: x_n=1\}.
				$ Assume that $\mathcal{A}$ has density $0$. Then  the set
		$
		\{n\in \N:  x_n \text{ is prime}\}
		$
		has density $0$.
	\end{theorem}

Apart from
some trivial examples,   Lucas sequences of the first kind  satisfy the hypothesis of the theorem, 
as do elliptic divisibility sequences.
 Moreover, the assumption on the density of $\mathcal{A}$ cannot be removed, as we will show in Remark \ref{rem:necX}. 
The proof of Theorem \ref{thm:cardprimes} relies 
on an auxiliary fact that is provided for us by Sandro Bettin in Appendix \ref{app}. This states that,
given any divisor closed set of density $0$, the product set formed with the set of primes continues to have density $0$.

As an application of our main results, let us record upper and lower bounds for the  cardinality of 
$z$-rough numbers in the  Fibonacci sequence  $\{F_n\}_{n\in \NN}$, with $n\leq N$ and $z=\frac{1}{2}(\log N)(\log\log N)$. (In fact, the  following result shows that  the set of $n\leq N$ for which $F_n$ is $z$-rough has density $0$.)

  \begin{corollary}\label{cor:fib}
  	Let $\{F_n\}_{n\in \NN}$ be the sequence of Fibonacci numbers. Then 
  	$$
  	\frac{N}{\log\log N}\ll \#\left\{n\leq N :
  	p\mid F_n \implies  p> \tfrac{1}{2}(\log N)(\log\log N) \right\}
  	\ll \frac{N}{\log\log\log N}.
  	$$
  \end{corollary}
  \begin{proof}
  Recall that the assumptions of Theorems \ref{thm:gen} and \ref{thm:upperf} hold for the Fibonacci sequence. 
  	The lower bound therefore follows from Theorem \ref{thm:gen}. For the upper bound, we put $z=\tfrac{1}{2}(\log N)(\log\log N)$ and deduce from Theorem \ref{thm:upperf}  that
  	$$
  	\#\left\{n\leq N :
  	p\mid F_n \implies  p> \tfrac{1}{2}(\log N)(\log\log N) \right\}\ll \frac{N}{\log f(z)}+f(z)^2.
  	$$
	If $f(z)=\log_\phi(\sqrt{5}z-1)$, then it is  clear that $F_n\leq z$ if $n\leq f(z)$. The desired upper bound is now obvious. 
  \end{proof}

Much of this paper  is devoted to understanding the size of the expected main term in the application of sieve methods to a SDS,  with the  associated quantities $m_d$ from Definition \ref{def:md}. 
In view of \eqref{eq:LOD}, an application of inclusion--exclusion confers a special significance to the sum 
	\begin{equation}\label{eq:Mz}
M(\Pi)=\sum_{d\mid \Pi}\frac{\mu(d)}{m_d},
	\end{equation}
	for given $\Pi\in \NN$.
	(If $m_d=\infty$, then we count $\mu(d)/m_d$ as $0$ in this sum.)
	We shall prove the following lower bound.

	\begin{theorem}\label{thm:lower}
			Let $\{x_n\}_{n\in \NN}$ be a SDS with the property that there 
		exists 
		 $\alpha>0$ 	 
		 such that 			$
			m_p<p^\alpha
			$
for all but finitely many primes.
Then $$
M(\Pi_{z})\gg  \frac{1}{\log z}$$
where $\Pi_z=\prod_{p<z} p$.
\end{theorem}

	Consider the SDS given by  Mersenne numbers $x_n=2^n-1$. In this case 
		computational evidence seems to suggest that
		$$
		M(\Pi_{z})\sim   \frac{1}{\log z},
		$$
		as $z\to \infty$. 
We have illustrated this in  Figure \ref{fig}. (In Remark \ref{rem:rat}  we will see that 
$M(\Pi_{z})\ll   \frac{1}{\log\log z}$, 
for Mersenne numbers.)

In \cite[Section~7]{granville23}, Grantham and Granville briefly discuss the challenges around  applying sieve methods to  sequences of the form $2^n-b$. One of their suggested modifications 
involves changing  the ordering  in $M(\Pi_{z})$, which is  
by the size of the prime divisors of $d$.  In 
 private communication with the authors, Granville has asked whether similar behaviour is expected for the sum
$$
\sum_{\substack{d\in \NN\\ 
m_d\leq z}}\frac{\mu(d)}{m_d},
$$
as $z\to \infty$.

		\begin{figure}
		\begin{tikzpicture}[xscale=1.5, yscale=1.5]
			\begin{axis}
				[ymin=0,ymax=0.8,xmin=0,xmax=600]
				\pgfmathsetlengthmacro\MajorTickLength{0.5}
				\pgfplotsset{every x tick label/.append style={font=\tiny, yshift=0.5ex}}
				\pgfplotsset{every y tick label/.append style={font=\tiny, xshift=0.5ex}}
				\addplot[blue, only marks, mark size=0.2mm] coordinates {
( 5 , 0.50000 )
( 7 , 0.33333 )
( 11 , 0.33333 )
( 13 , 0.33333 )
( 17 , 0.33333 )
( 19 , 0.33333 )
( 23 , 0.30303 )
( 29 , 0.30303 )
( 31 , 0.24242 )
( 37 , 0.24242 )
( 41 , 0.24242 )
( 43 , 0.24242 )
( 47 , 0.23188 )
( 53 , 0.23188 )
( 59 , 0.23188 )
( 61 , 0.23188 )
( 67 , 0.23188 )
( 71 , 0.23188 )
( 73 , 0.23188 )
( 79 , 0.23188 )
( 83 , 0.23188 )
( 89 , 0.23188 )
( 97 , 0.23188 )
( 101 , 0.23188 )
( 103 , 0.23188 )
( 107 , 0.23188 )
( 109 , 0.23188 )
( 113 , 0.23188 )
( 127 , 0.19876 )
( 131 , 0.19876 )
( 137 , 0.19876 )
( 139 , 0.19876 )
( 149 , 0.19876 )
( 151 , 0.19876 )
( 157 , 0.19876 )
( 163 , 0.19876 )
( 167 , 0.19636 )
( 173 , 0.19636 )
( 179 , 0.19636 )
( 181 , 0.19636 )
( 191 , 0.19636 )
( 193 , 0.19636 )
( 197 , 0.19636 )
( 199 , 0.19636 )
( 211 , 0.19636 )
( 223 , 0.19106 )
( 227 , 0.19106 )
( 229 , 0.19106 )
( 233 , 0.18447 )
( 239 , 0.18447 )
( 241 , 0.18447 )
( 251 , 0.18447 )
( 257 , 0.18447 )
( 263 , 0.18306 )
( 269 , 0.18306 )
( 271 , 0.18306 )
( 277 , 0.18306 )
( 281 , 0.18306 )
( 283 , 0.18306 )
( 293 , 0.18306 )
( 307 , 0.18306 )
( 311 , 0.18306 )
( 313 , 0.18306 )
( 317 , 0.18306 )
( 331 , 0.18306 )
( 337 , 0.18306 )
( 347 , 0.18306 )
( 349 , 0.18306 )
( 353 , 0.18306 )
( 359 , 0.18204 )
( 367 , 0.18204 )
( 373 , 0.18204 )
( 379 , 0.18204 )
( 383 , 0.18108 )
( 389 , 0.18108 )
( 397 , 0.18108 )
( 401 , 0.18108 )
( 409 , 0.18108 )
( 419 , 0.18108 )
( 421 , 0.18108 )
( 431 , 0.17687 )
( 433 , 0.17687 )
( 439 , 0.17445 )
( 443 , 0.17445 )
( 449 , 0.17445 )
( 457 , 0.17445 )
( 461 , 0.17445 )
( 463 , 0.17445 )
( 467 , 0.17445 )
( 479 , 0.17372 )
( 487 , 0.17372 )
( 491 , 0.17372 )
( 499 , 0.17372 )
( 503 , 0.17303 )
( 509 , 0.17303 )
( 521 , 0.17303 )
( 523 , 0.17303 )
( 541 , 0.17303 )
( 547 , 0.17303 )
( 557 , 0.17303 )
( 563 , 0.17303 )
( 569 , 0.17303 )
( 571 , 0.17303 )
( 577 , 0.17303 )
( 587 , 0.17303 )
( 593 , 0.17303 )

				};
				\addplot[red] coordinates {
( 4 , 0.72134 )
( 5 , 0.62134 )
( 6 , 0.55811 )
( 7 , 0.51390 )
( 8 , 0.48090 )
( 9 , 0.45512 )
( 10 , 0.43430 )
( 11 , 0.41703 )
( 12 , 0.40243 )
( 13 , 0.38987 )
( 14 , 0.37892 )
( 15 , 0.36927 )
( 16 , 0.36067 )
( 17 , 0.35295 )
( 18 , 0.34598 )
( 19 , 0.33962 )
( 20 , 0.33381 )
( 21 , 0.32846 )
( 22 , 0.32352 )
( 23 , 0.31893 )
( 24 , 0.31466 )
( 25 , 0.31067 )
( 26 , 0.30693 )
( 27 , 0.30341 )
( 28 , 0.30010 )
( 29 , 0.29697 )
( 30 , 0.29401 )
( 31 , 0.29121 )
( 32 , 0.28854 )
( 33 , 0.28600 )
( 34 , 0.28358 )
( 35 , 0.28127 )
( 36 , 0.27906 )
( 37 , 0.27694 )
( 38 , 0.27491 )
( 39 , 0.27296 )
( 40 , 0.27108 )
( 41 , 0.26928 )
( 42 , 0.26755 )
( 43 , 0.26587 )
( 44 , 0.26426 )
( 45 , 0.26270 )
( 46 , 0.26119 )
( 47 , 0.25973 )
( 48 , 0.25832 )
( 49 , 0.25695 )
( 50 , 0.25562 )
( 51 , 0.25433 )
( 52 , 0.25309 )
( 53 , 0.25187 )
( 54 , 0.25069 )
( 55 , 0.24954 )
( 56 , 0.24843 )
( 57 , 0.24734 )
( 58 , 0.24628 )
( 59 , 0.24525 )
( 60 , 0.24424 )
( 61 , 0.24326 )
( 62 , 0.24230 )
( 63 , 0.24136 )
( 64 , 0.24045 )
( 65 , 0.23956 )
( 66 , 0.23868 )
( 67 , 0.23783 )
( 68 , 0.23699 )
( 69 , 0.23618 )
( 70 , 0.23538 )
( 71 , 0.23459 )
( 72 , 0.23383 )
( 73 , 0.23308 )
( 74 , 0.23234 )
( 75 , 0.23162 )
( 76 , 0.23091 )
( 77 , 0.23021 )
( 78 , 0.22953 )
( 79 , 0.22886 )
( 80 , 0.22820 )
( 81 , 0.22756 )
( 82 , 0.22693 )
( 83 , 0.22630 )
( 84 , 0.22569 )
( 85 , 0.22509 )
( 86 , 0.22450 )
( 87 , 0.22392 )
( 88 , 0.22335 )
( 89 , 0.22278 )
( 90 , 0.22223 )
( 91 , 0.22169 )
( 92 , 0.22115 )
( 93 , 0.22062 )
( 94 , 0.22010 )
( 95 , 0.21959 )
( 96 , 0.21909 )
( 97 , 0.21859 )
( 98 , 0.21810 )
( 99 , 0.21762 )
( 100 , 0.21715 )
( 101 , 0.21668 )
( 102 , 0.21622 )
( 103 , 0.21576 )
( 104 , 0.21531 )
( 105 , 0.21487 )
( 106 , 0.21443 )
( 107 , 0.21400 )
( 108 , 0.21358 )
( 109 , 0.21316 )
( 110 , 0.21274 )
( 111 , 0.21234 )
( 112 , 0.21193 )
( 113 , 0.21153 )
( 114 , 0.21114 )
( 115 , 0.21075 )
( 116 , 0.21037 )
( 117 , 0.20999 )
( 118 , 0.20961 )
( 119 , 0.20924 )
( 120 , 0.20888 )
( 121 , 0.20852 )
( 122 , 0.20816 )
( 123 , 0.20781 )
( 124 , 0.20746 )
( 125 , 0.20711 )
( 126 , 0.20677 )
( 127 , 0.20643 )
( 128 , 0.20610 )
( 129 , 0.20577 )
( 130 , 0.20544 )
( 131 , 0.20512 )
( 132 , 0.20480 )
( 133 , 0.20448 )
( 134 , 0.20417 )
( 135 , 0.20386 )
( 136 , 0.20356 )
( 137 , 0.20325 )
( 138 , 0.20295 )
( 139 , 0.20266 )
( 140 , 0.20236 )
( 141 , 0.20207 )
( 142 , 0.20178 )
( 143 , 0.20150 )
( 144 , 0.20122 )
( 145 , 0.20094 )
( 146 , 0.20066 )
( 147 , 0.20038 )
( 148 , 0.20011 )
( 149 , 0.19984 )
( 150 , 0.19958 )
( 151 , 0.19931 )
( 152 , 0.19905 )
( 153 , 0.19879 )
( 154 , 0.19853 )
( 155 , 0.19828 )
( 156 , 0.19802 )
( 157 , 0.19777 )
( 158 , 0.19753 )
( 159 , 0.19728 )
( 160 , 0.19704 )
( 161 , 0.19680 )
( 162 , 0.19656 )
( 163 , 0.19632 )
( 164 , 0.19608 )
( 165 , 0.19585 )
( 166 , 0.19562 )
( 167 , 0.19539 )
( 168 , 0.19516 )
( 169 , 0.19493 )
( 170 , 0.19471 )
( 171 , 0.19449 )
( 172 , 0.19427 )
( 173 , 0.19405 )
( 174 , 0.19383 )
( 175 , 0.19362 )
( 176 , 0.19341 )
( 177 , 0.19319 )
( 178 , 0.19298 )
( 179 , 0.19278 )
( 180 , 0.19257 )
( 181 , 0.19236 )
( 182 , 0.19216 )
( 183 , 0.19196 )
( 184 , 0.19176 )
( 185 , 0.19156 )
( 186 , 0.19136 )
( 187 , 0.19116 )
( 188 , 0.19097 )
( 189 , 0.19078 )
( 190 , 0.19058 )
( 191 , 0.19039 )
( 192 , 0.19020 )
( 193 , 0.19002 )
( 194 , 0.18983 )
( 195 , 0.18965 )
( 196 , 0.18946 )
( 197 , 0.18928 )
( 198 , 0.18910 )
( 199 , 0.18892 )
( 200 , 0.18874 )
( 201 , 0.18856 )
( 202 , 0.18839 )
( 203 , 0.18821 )
( 204 , 0.18804 )
( 205 , 0.18786 )
( 206 , 0.18769 )
( 207 , 0.18752 )
( 208 , 0.18735 )
( 209 , 0.18718 )
( 210 , 0.18702 )
( 211 , 0.18685 )
( 212 , 0.18669 )
( 213 , 0.18652 )
( 214 , 0.18636 )
( 215 , 0.18620 )
( 216 , 0.18604 )
( 217 , 0.18588 )
( 218 , 0.18572 )
( 219 , 0.18556 )
( 220 , 0.18540 )
( 221 , 0.18525 )
( 222 , 0.18509 )
( 223 , 0.18494 )
( 224 , 0.18479 )
( 225 , 0.18464 )
( 226 , 0.18448 )
( 227 , 0.18433 )
( 228 , 0.18419 )
( 229 , 0.18404 )
( 230 , 0.18389 )
( 231 , 0.18374 )
( 232 , 0.18360 )
( 233 , 0.18345 )
( 234 , 0.18331 )
( 235 , 0.18316 )
( 236 , 0.18302 )
( 237 , 0.18288 )
( 238 , 0.18274 )
( 239 , 0.18260 )
( 240 , 0.18246 )
( 241 , 0.18232 )
( 242 , 0.18218 )
( 243 , 0.18205 )
( 244 , 0.18191 )
( 245 , 0.18178 )
( 246 , 0.18164 )
( 247 , 0.18151 )
( 248 , 0.18138 )
( 249 , 0.18124 )
( 250 , 0.18111 )
( 251 , 0.18098 )
( 252 , 0.18085 )
( 253 , 0.18072 )
( 254 , 0.18059 )
( 255 , 0.18046 )
( 256 , 0.18034 )
( 257 , 0.18021 )
( 258 , 0.18008 )
( 259 , 0.17996 )
( 260 , 0.17983 )
( 261 , 0.17971 )
( 262 , 0.17959 )
( 263 , 0.17946 )
( 264 , 0.17934 )
( 265 , 0.17922 )
( 266 , 0.17910 )
( 267 , 0.17898 )
( 268 , 0.17886 )
( 269 , 0.17874 )
( 270 , 0.17862 )
( 271 , 0.17850 )
( 272 , 0.17839 )
( 273 , 0.17827 )
( 274 , 0.17815 )
( 275 , 0.17804 )
( 276 , 0.17792 )
( 277 , 0.17781 )
( 278 , 0.17769 )
( 279 , 0.17758 )
( 280 , 0.17747 )
( 281 , 0.17736 )
( 282 , 0.17724 )
( 283 , 0.17713 )
( 284 , 0.17702 )
( 285 , 0.17691 )
( 286 , 0.17680 )
( 287 , 0.17669 )
( 288 , 0.17659 )
( 289 , 0.17648 )
( 290 , 0.17637 )
( 291 , 0.17626 )
( 292 , 0.17616 )
( 293 , 0.17605 )
( 294 , 0.17595 )
( 295 , 0.17584 )
( 296 , 0.17574 )
( 297 , 0.17563 )
( 298 , 0.17553 )
( 299 , 0.17542 )
( 300 , 0.17532 )
( 301 , 0.17522 )
( 302 , 0.17512 )
( 303 , 0.17502 )
( 304 , 0.17492 )
( 305 , 0.17482 )
( 306 , 0.17472 )
( 307 , 0.17462 )
( 308 , 0.17452 )
( 309 , 0.17442 )
( 310 , 0.17432 )
( 311 , 0.17422 )
( 312 , 0.17413 )
( 313 , 0.17403 )
( 314 , 0.17393 )
( 315 , 0.17384 )
( 316 , 0.17374 )
( 317 , 0.17365 )
( 318 , 0.17355 )
( 319 , 0.17345 )
( 320 , 0.17336 )
( 321 , 0.17327 )
( 322 , 0.17317 )
( 323 , 0.17308 )
( 324 , 0.17299 )
( 325 , 0.17290 )
( 326 , 0.17280 )
( 327 , 0.17271 )
( 328 , 0.17262 )
( 329 , 0.17253 )
( 330 , 0.17244 )
( 331 , 0.17235 )
( 332 , 0.17226 )
( 333 , 0.17217 )
( 334 , 0.17208 )
( 335 , 0.17200 )
( 336 , 0.17191 )
( 337 , 0.17182 )
( 338 , 0.17173 )
( 339 , 0.17164 )
( 340 , 0.17156 )
( 341 , 0.17147 )
( 342 , 0.17138 )
( 343 , 0.17130 )
( 344 , 0.17121 )
( 345 , 0.17113 )
( 346 , 0.17104 )
( 347 , 0.17096 )
( 348 , 0.17088 )
( 349 , 0.17079 )
( 350 , 0.17071 )
( 351 , 0.17063 )
( 352 , 0.17054 )
( 353 , 0.17046 )
( 354 , 0.17038 )
( 355 , 0.17030 )
( 356 , 0.17022 )
( 357 , 0.17013 )
( 358 , 0.17005 )
( 359 , 0.16997 )
( 360 , 0.16989 )
( 361 , 0.16981 )
( 362 , 0.16973 )
( 363 , 0.16965 )
( 364 , 0.16957 )
( 365 , 0.16949 )
( 366 , 0.16942 )
( 367 , 0.16934 )
( 368 , 0.16926 )
( 369 , 0.16918 )
( 370 , 0.16910 )
( 371 , 0.16903 )
( 372 , 0.16895 )
( 373 , 0.16887 )
( 374 , 0.16880 )
( 375 , 0.16872 )
( 376 , 0.16865 )
( 377 , 0.16857 )
( 378 , 0.16850 )
( 379 , 0.16842 )
( 380 , 0.16834 )
( 381 , 0.16827 )
( 382 , 0.16820 )
( 383 , 0.16812 )
( 384 , 0.16805 )
( 385 , 0.16798 )
( 386 , 0.16790 )
( 387 , 0.16783 )
( 388 , 0.16776 )
( 389 , 0.16768 )
( 390 , 0.16761 )
( 391 , 0.16754 )
( 392 , 0.16747 )
( 393 , 0.16740 )
( 394 , 0.16733 )
( 395 , 0.16726 )
( 396 , 0.16718 )
( 397 , 0.16711 )
( 398 , 0.16704 )
( 399 , 0.16697 )
( 400 , 0.16690 )
( 401 , 0.16683 )
( 402 , 0.16677 )
( 403 , 0.16670 )
( 404 , 0.16663 )
( 405 , 0.16656 )
( 406 , 0.16649 )
( 407 , 0.16642 )
( 408 , 0.16636 )
( 409 , 0.16629 )
( 410 , 0.16622 )
( 411 , 0.16615 )
( 412 , 0.16608 )
( 413 , 0.16602 )
( 414 , 0.16595 )
( 415 , 0.16588 )
( 416 , 0.16582 )
( 417 , 0.16575 )
( 418 , 0.16569 )
( 419 , 0.16562 )
( 420 , 0.16556 )
( 421 , 0.16549 )
( 422 , 0.16543 )
( 423 , 0.16536 )
( 424 , 0.16530 )
( 425 , 0.16523 )
( 426 , 0.16517 )
( 427 , 0.16510 )
( 428 , 0.16504 )
( 429 , 0.16498 )
( 430 , 0.16491 )
( 431 , 0.16485 )
( 432 , 0.16479 )
( 433 , 0.16472 )
( 434 , 0.16466 )
( 435 , 0.16460 )
( 436 , 0.16454 )
( 437 , 0.16448 )
( 438 , 0.16441 )
( 439 , 0.16435 )
( 440 , 0.16429 )
( 441 , 0.16423 )
( 442 , 0.16417 )
( 443 , 0.16411 )
( 444 , 0.16405 )
( 445 , 0.16399 )
( 446 , 0.16393 )
( 447 , 0.16387 )
( 448 , 0.16381 )
( 449 , 0.16375 )
( 450 , 0.16369 )
( 451 , 0.16363 )
( 452 , 0.16357 )
( 453 , 0.16351 )
( 454 , 0.16345 )
( 455 , 0.16339 )
( 456 , 0.16333 )
( 457 , 0.16327 )
( 458 , 0.16322 )
( 459 , 0.16316 )
( 460 , 0.16310 )
( 461 , 0.16304 )
( 462 , 0.16298 )
( 463 , 0.16293 )
( 464 , 0.16287 )
( 465 , 0.16281 )
( 466 , 0.16276 )
( 467 , 0.16270 )
( 468 , 0.16264 )
( 469 , 0.16259 )
( 470 , 0.16253 )
( 471 , 0.16247 )
( 472 , 0.16242 )
( 473 , 0.16236 )
( 474 , 0.16231 )
( 475 , 0.16225 )
( 476 , 0.16220 )
( 477 , 0.16214 )
( 478 , 0.16208 )
( 479 , 0.16203 )
( 480 , 0.16198 )
( 481 , 0.16192 )
( 482 , 0.16187 )
( 483 , 0.16181 )
( 484 , 0.16176 )
( 485 , 0.16170 )
( 486 , 0.16165 )
( 487 , 0.16160 )
( 488 , 0.16154 )
( 489 , 0.16149 )
( 490 , 0.16144 )
( 491 , 0.16138 )
( 492 , 0.16133 )
( 493 , 0.16128 )
( 494 , 0.16122 )
( 495 , 0.16117 )
( 496 , 0.16112 )
( 497 , 0.16107 )
( 498 , 0.16101 )
( 499 , 0.16096 )
( 500 , 0.16091 )
( 501 , 0.16086 )
( 502 , 0.16081 )
( 503 , 0.16076 )
( 504 , 0.16071 )
( 505 , 0.16065 )
( 506 , 0.16060 )
( 507 , 0.16055 )
( 508 , 0.16050 )
( 509 , 0.16045 )
( 510 , 0.16040 )
( 511 , 0.16035 )
( 512 , 0.16030 )
( 513 , 0.16025 )
( 514 , 0.16020 )
( 515 , 0.16015 )
( 516 , 0.16010 )
( 517 , 0.16005 )
( 518 , 0.16000 )
( 519 , 0.15995 )
( 520 , 0.15990 )
( 521 , 0.15985 )
( 522 , 0.15980 )
( 523 , 0.15976 )
( 524 , 0.15971 )
( 525 , 0.15966 )
( 526 , 0.15961 )
( 527 , 0.15956 )
( 528 , 0.15951 )
( 529 , 0.15946 )
( 530 , 0.15942 )
( 531 , 0.15937 )
( 532 , 0.15932 )
( 533 , 0.15927 )
( 534 , 0.15923 )
( 535 , 0.15918 )
( 536 , 0.15913 )
( 537 , 0.15908 )
( 538 , 0.15904 )
( 539 , 0.15899 )
( 540 , 0.15894 )
( 541 , 0.15890 )
( 542 , 0.15885 )
( 543 , 0.15880 )
( 544 , 0.15876 )
( 545 , 0.15871 )
( 546 , 0.15866 )
( 547 , 0.15862 )
( 548 , 0.15857 )
( 549 , 0.15853 )
( 550 , 0.15848 )
( 551 , 0.15844 )
( 552 , 0.15839 )
( 553 , 0.15834 )
( 554 , 0.15830 )
( 555 , 0.15825 )
( 556 , 0.15821 )
( 557 , 0.15816 )
( 558 , 0.15812 )
( 559 , 0.15807 )
( 560 , 0.15803 )
( 561 , 0.15799 )
( 562 , 0.15794 )
( 563 , 0.15790 )
( 564 , 0.15785 )
( 565 , 0.15781 )
( 566 , 0.15776 )
( 567 , 0.15772 )
( 568 , 0.15768 )
( 569 , 0.15763 )
( 570 , 0.15759 )
( 571 , 0.15755 )
( 572 , 0.15750 )
( 573 , 0.15746 )
( 574 , 0.15742 )
( 575 , 0.15737 )
( 576 , 0.15733 )
( 577 , 0.15729 )
( 578 , 0.15724 )
( 579 , 0.15720 )
( 580 , 0.15716 )
( 581 , 0.15712 )
( 582 , 0.15707 )
( 583 , 0.15703 )
( 584 , 0.15699 )
( 585 , 0.15695 )
( 586 , 0.15690 )
( 587 , 0.15686 )
( 588 , 0.15682 )
( 589 , 0.15678 )
( 590 , 0.15674 )
( 591 , 0.15669 )
( 592 , 0.15665 )
( 593 , 0.15661 )
};
			\end{axis}
		\end{tikzpicture}
		\caption
		{The blue graph is the graph of the function $M(\Pi_{z})$ for Mersenne numbers, the red one is the graph of $\frac{1}{\log z}$.}
		\label{fig}
	\end{figure}
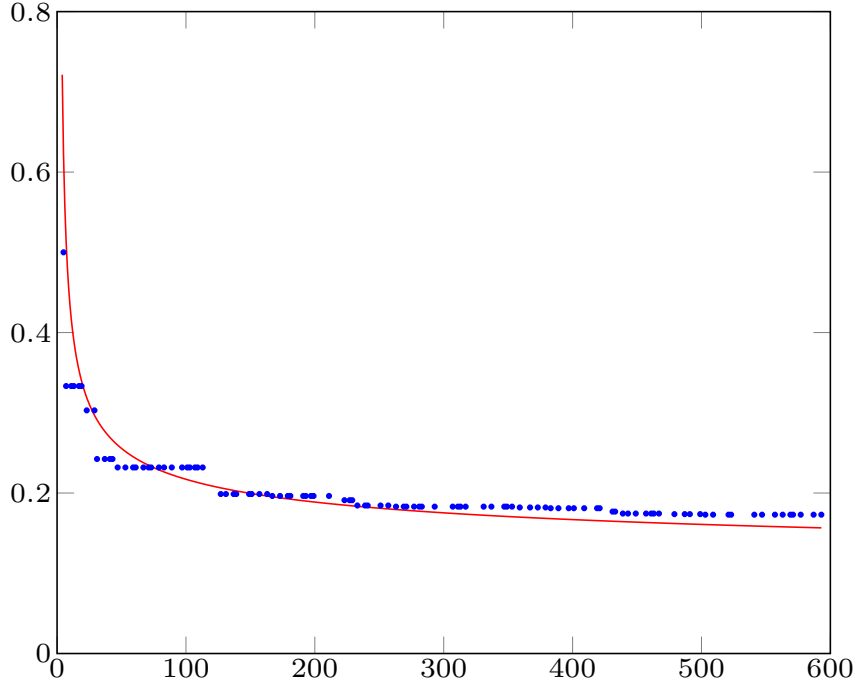

\medskip

Let us close this introduction by outlining the contents of the paper.
Firstly, in Section \ref{sec:propmd} we shall study some properties of $m_d$ and prove 
Theorem \ref{thm:lower}.
With this to hand, the proof of 
Theorems \ref{thm:gen}, \ref{thm:upperf}, and 
\ref{thm:cardprimes} 
will be carried out in 
Section \ref{sec:E-L}, although the latter does not use any sieve theory.  
In Section 
\ref{sec:eds} we shall specialise to elliptic divisibility sequences, and investigate $M(\Pi)$ under additional hypotheses in Sections \ref{sec:cheb}  and \ref{sec:koblitz}.
In particular, in Theorem \ref{thm:finalkob}, we shall  prove 
a refinement of Theorem \ref{thm:gen} for 
 non-CM elliptic curves, subject to some standard hypotheses.
 Finally, in Appendix \ref{app}, 
 Sandro Bettin has supplied a useful fact about  
 the density of the product of 
 the set of primes with 
 a divisor closed set.

\subsection*{Acknowledgements}
The authors 
are very grateful to 
Andrew Granville, 
Dimitris Koukoulopoulos, Davide Lombardo, Florian Luca,
Igor Shparlinski,  and Joni Ter\"av\"ainen for useful comments.
While working on this paper
the first author was supported  by 
a FWF grant (DOI 10.55776/P32428)
and 
the second author was  supported by the European Union’s Horizon 2020 research and 
innovation program under the Marie Skłodowska-Curie Grant Agreement No. 
101034413.
	
		\section{Properties of $m_d$}\label{sec:propmd}
	Let $\{x_n\}_{n\in \NN}$ be a strong divisibility sequence (SDS), so that 
		$\gcd(x_n,x_m)=x_{\gcd(n,m)}$ 
for all  $m,n\in \NN$.  Clearly this implies that $x_m\mid x_n$ if $m\mid n$, so that $\{x_n\}_{n\in \NN}$ is a {\em divisibility sequence}. As we pointed out in the introduction, we will work under the assumption $x_1=1$; the general case can be deduced from this case by applying the results to the sequence $x_n/x_1$. We begin by collecting together some basic properties of the function $m_d$ from Definition \ref{def:md}.

		\begin{lemma}\label{lemma:mddivniff}
		We have $d\mid x_n$ if and only if $m_d\mid n$.
	\end{lemma}
	\begin{proof}
		If $m_d\mid n$, then $d\mid x_{m_d}\mid x_n$ since we are working with divisibility sequences. If $d\mid x_n$, then $d\mid \gcd(x_{m_d},x_n)=x_{\gcd(n,m_d)}$. By definition, $m_d$ is the smallest positive integer such that $x_{m_d}$ is divisible by $d$, whence $\gcd(n,m_d)\geq m_d$. This happens only if $m_d\mid n$.
	\end{proof}
	
	\begin{remark}\label{rem:LOD}
	Note that  \eqref{eq:LOD} is an immediate consequence of Lemma 
	\ref{lemma:mddivniff}.
	\end{remark}
	
	\begin{lemma}\label{lemma:mdivisibility}
		If $k\mid j$ then $m_k\mid m_j$.
	\end{lemma}
	\begin{proof}
Let   $k\mid j$. Then $k\mid j\mid x_{m_j}$, whence 
		 Lemma \ref{lemma:mddivniff} implies that
		$m_k\mid m_j$.
	\end{proof}
	
	\begin{lemma}\label{lemma:lcm}
		Let $d_1,d_2\in \NN$. Then
		$$
		m_{[d_1,d_2]}=[m_{d_1},m_{d_2}],
		$$
		where $[\cdot,\cdot]$ denotes the least common multiple.
	\end{lemma}
	\begin{proof}
		Since $d_1\mid [d_1,d_2]$, we have $m_{d_1}\mid m_{[d_1,d_2]}$ by Lemma \ref{lemma:mdivisibility}. In the same way, $m_{d_2}\mid m_{[d_1,d_2]}$ and so $[m_{d_1},m_{d_2}]\mid m_{[d_1,d_2]}$.
		
		The sequence $\{x_n\}_{n\in \N}$ is a divisibility sequence and so $x_{m_{d_1}}\mid x_{[m_{d_1},m_{d_2}]}$. It follows that $d_1\mid x_{m_{d_1}}\mid  x_{[m_{d_1},m_{d_2}]}$ and, similarly,  $d_2\mid  x_{[m_{d_1},m_{d_2}]}$. Thus, $[d_1,d_2]\mid x_{[m_{d_1},m_{d_2}]}$ and
		Lemma \ref{lemma:mddivniff} yields $m_{[d_1,d_2]}\mid [m_{d_1},m_{d_2}]$.
			\end{proof}

We now turn to the estimation of the sum $M(\Pi)$ in 
\eqref{eq:Mz}, for suitable $\Pi\in \NN$. Our first  goal is to prove the lower bound in Theorem \ref{thm:lower}. For this we shall apply  a suitable result based on  inclusion--exclusion. 

	\begin{lemma}\label{lemma:incexc}
		Let $n\in \NN$ and  let $n_1,\dots, n_k$ be  divisors of $n$. 		
		Then
		$$
		\sum_{\substack{d\mid n\\d\nmid n_j, \forall j\leq k}}\mu(d)=\sum_{I\subseteq\{1,\dots k\}\cup \emptyset}(-1)^{\#I}\sum_{\substack{d\mid \gcd(n_i)_{i\in I}}}\mu(d),
		$$
		where we follow the convention that $\gcd(n_i)_{i\in I}=n$ when $I=\emptyset$.
	\end{lemma}
	\begin{proof}
		We proceed by induction on $k$. Assume $k=1$. Then 
		$$
		\sum_{\substack{d\mid n\\d\nmid n_1}}\mu(d)=\sum_{\substack{d\mid n}}\mu(d)-\sum_{\substack{d\mid n_1}}\mu(d),
		$$
		as required. We now  prove the lemma for $k>1$. Let $D=\{d\mid n: d\nmid n_j, \forall j\leq k-1\}$. Then we may write
		$$
		\sum_{\substack{d\mid n\\d\nmid n_j, \forall j\leq k}}\mu(d)=
		\sum_{\substack{d\in D\\ d\nmid n_k}}\mu(d)=
		\sum_{\substack{d\in D}}\mu(d)- 
		\sum_{\substack{d\in D\\ d\mid n_k}}\mu(d).
		$$		
		By induction,
		$$
		\sum_{\substack{d\in D}}\mu(d)=\sum_{I\subseteq\{1,\dots k-1\}\cup \emptyset}
		(-1)^{\#I}
		\sum_{\substack{d\mid \gcd(n_i)_{i\in I}}}\mu(d).
		$$
		Note that if $d\in D$ is such that $d\mid n_k$, then $d\mid n_k$ and $d\nmid \gcd(n_k,n_i)$ for each $1\leq i\leq k-1$. Hence, by induction, we have 
		$$
		\sum_{\substack{d\in D\\d\mid n_k}}\mu(d)=\sum_{\substack{d\mid n_k\\d\nmid \gcd(n_j,n_k),
				\forall j\leq k-1}}\mu(d)=
		\sum_{I\subseteq \{1,\dots k-1\}\cup \emptyset}(-1)^{\#I}\sum_{\substack{d\mid \gcd(n_i,n_k)_{i\in I}}}\mu(d).
		$$
		The statement of the lemma follows. 
	\end{proof}
	
	Next, the following key result allows us to express $M(\Pi)$ as a sum of non-negative terms. 
	
	\begin{lemma}\label{lemma:summd}
		Let $\Pi\in \NN$ be square-free and such that $m_\Pi<\infty$. Then
		$$
M(\Pi)=\sum_{\substack{j\mid m_\Pi \\\gcd(x_j,\Pi)=1}}\frac{\varphi(m_\Pi/j)}{m_\Pi},
		$$
		where $\varphi(\cdot)$ is the Euler $\varphi$-function.
	\end{lemma}
	\begin{proof}
		Recall from Lemma \ref{lemma:mdivisibility} that $m_d\mid m_\Pi$ if $d\mid \Pi$. Thus
		$$
		M(\Pi)=
		\sum_{d\mid \Pi}\frac{\mu(d)}{m_d}=\sum_{l\mid m_\Pi}\frac{1}{l}
		\sum_{\substack{m_d=l\\d\mid \Pi}}\mu(d).
		$$
		Now  $m_d=l$ if and only if we have $d\mid x_l$, with  $d\nmid x_{l'}$ for each proper divisor $l'\mid l$. 
		Let $p_1,\dots p_k$ be the distinct prime divisors of $l$. It follows that $m_d=l$ if and only if $d\mid x_l$ and $d\nmid x_{l/p_i}$, for each $1\leq i\leq k$. 
		Hence
		$$
		\sum_{\substack{m_d=l\\d\mid \Pi}}\mu(d)=
		\sum_{\substack{d\mid \gcd(x_l,\Pi)\\d\nmid \gcd(x_{l/p},\Pi), \forall p\mid l}}\mu(d).
		$$
		We now seek to 
		Lemma \ref{lemma:incexc}, for which we 
		note that 
		$$
		\gcd(x_{l/p_i})_{i\in I}=x_{\gcd(l/p_i)_{i\in I}}=x_{l/\prod_{i\in I}p_i},
		$$ 
		since $\{x_n\}_{n\in \NN}$ is a SDS.
		Hence
		$$
		\sum_{\substack{m_d=l\\d\mid \Pi}}\mu(d)=
		\sum_{I\subseteq\{1,\dots k\}\cup \emptyset}(-1)^{\#I}\sum_{d\mid \gcd(x_{l/p_I},\Pi)}\mu(d),
		$$
		where $p_1,\dots p_k$ are the distinct prime divisors of $l$ and $p_I=\prod_{i\in I}p_i$. 
		The inner sum is $1$ if $\gcd(x_{l/p_I},\Pi)=1$, and $0$ otherwise. 
		It therefore follows that 
		$$
		\sum_{\substack{m_d=l\\d\mid \Pi}}\mu(d)=\sum_{\substack{k\mid l\\\gcd(x_{l/k},\Pi)=1}}\mu(k).
		$$
		Putting $l=jk$, we see that 
		$$
		M(\Pi)=
		\sum_{l\mid m_\Pi}\frac{1}{l}	
		\sum_{\substack{l=jk\\ \gcd(x_j,\Pi)=1}}
		\mu(k)
		=\sum_{\substack{j\mid m_\Pi \\\gcd(x_j,\Pi)=1}}\frac{1}{j}\sum_{k\mid m_\Pi/j}\frac{\mu(k)}{k}.
		$$
		By the properties of the Euler-$\varphi$ function, we have 
		$$
		M(\Pi)=\sum_{\substack{j\mid m_\Pi \\\gcd(x_j,\Pi)=1}}\frac{1}{j}\frac{\varphi(m_\Pi/j)}{m_\Pi/j}=\sum_{\substack{j\mid m_\Pi \\\gcd(x_j,\Pi)=1}}\frac{\varphi(m_\Pi/j)}{m_\Pi},
		$$
		as desired. 
	\end{proof}
	
When $x_1=1$ the first term in this sum is $\varphi(m_\Pi)/m_\Pi$ and we get a lower bound for $M(\Pi)$ by focusing on this term and  ignoring the terms associated to $j>1$.
(Note that for the trivial sequence  $x_n=n$ this  is actually sharp, since then
		$$
		M(\Pi)=	\sum_{d\mid \Pi}\frac{\mu(d)}{d}=\frac{\varphi(\Pi)}{\Pi},
		$$
		for any $\Pi\in \NN$.)
 The following result is concerned with lower bounding
 $\varphi(m_\Pi)/m_\Pi$, under an additional assumption,   for suitable $\Pi\in \NN$.
	
	\begin{lemma}\label{lemma:upboundmPi}
		Let $\{x_n\}_{n\in \NN}$ be a SDS with the property that there 
		exists $\alpha>0$
such that 			$
m_p<p^\alpha
$
for all but finitely many primes. Let $S$ be the finite set of primes $p$ such that $m_p=\infty$. Let $\Pi_{z,S}=\prod_{p<z, p\notin S} p$. Then 
		$$
		\frac{\varphi(m_{\Pi_{z,S}})}{m_{\Pi_{z,S}}} \gg 
	\frac{1}{\log z}.
		$$
	\end{lemma}
	\begin{proof} On enlarging $\alpha$, we can assume that $m_p<p^\alpha$ for all $p\notin S$.
	Let us write $\Pi=\Pi_{z,S}$ to ease notation, noticing that 
$$
\log(\Pi)=\sum_{p\leq z}\log p+O(1)=\theta(z)+O(1),
$$ 
where $\theta(z)$ is the Chebyshev function. Thus $\log(\Pi)\sim z$ as $z\to \infty$, by the prime number theorem. 
			We have the familiar lower bound
	$$
	\varphi(d)\gg \frac{d}{\log\log d}.
	$$
	By Lemma \ref{lemma:lcm} and  the hypotheses of the lemma, it follows that 
	$$
	m_\Pi\leq \prod_{\substack{p\leq z\\p\notin S}} m_p<\prod_{\substack{p\leq z\\p\notin S}} p^\alpha=\Pi^\alpha.
	$$
	But then 
	$$
	\log\log m_\Pi\leq \log \alpha +\log\log \Pi\leq \log z +O(1).
	$$
	We have therefore proved that 
	$$
	\frac{\varphi(m_\Pi)}{m_\Pi}\gg \frac 1{\log\log m_\Pi} \gg \frac{1}{\log z},
	$$
	as required. 
	\end{proof}
	
	It turns out that the 
	 lower bound in 
	Lemma \ref{lemma:upboundmPi}
 is actually  sharp under a further assumption on the SDS. This is the object of the following  result.

	\begin{lemma}\label{lemma:lowboundmPi}
	Let $\{x_n\}_{n\in \NN}$ be a SDS. 
	Let $S$ be the finite set of primes $p$ such that $m_p=\infty$.
		Assume that there exists $\delta>0$ such that, for all but finitely many primes $q$, there exists a prime $p<q^\delta$ such that $q\mid m_p$. Then 
		$$
		\frac{\varphi(m_{\Pi_{z,S}})}{m_{\Pi_{z,S}}} \ll \frac{1}{\log z}.
		$$
	\end{lemma}

	\begin{proof}
Write $\Pi=\Pi_{z,S}$, as previously and note that 
		$$
		\frac{\varphi(m_\Pi)}{m_\Pi}=
		\prod_{\substack{q\mid m_\Pi}}
		\left(1-\frac{1}{q}\right).
		$$
		The product will only get larger if we reduce the number of factors in the product.  Now 
		$q\mid m_\Pi$ if and only if there exists $p\mid \Pi$ such that $q\mid m_p$.
		By hypothesis, we have $q\mid m_p$ for at least one $p\mid \Pi$ for $q\ll z^{1/\delta}$,  up to finitely many exceptions. Hence
		$$
		\frac{\varphi(m_\Pi)}{m_\Pi}=
		\prod_{\substack{q\mid m_\Pi}}
		\left(1-\frac{1}{q}\right)\ll
		\prod_{q \ll z^{1/\delta}}
		\left(1-\frac{1}{q}\right)
		\ll \frac{1}{\log z},
		$$
		by Mertens theorem. 
	\end{proof}
	
	\begin{remark}
		Both Lucas sequences of the first kind and CM elliptic divisibility sequences satisfy the assumption of Lemma \ref{lemma:lowboundmPi}, assuming GRH. For elliptic divisibility sequences  with CM, this will be proved in Proposition \ref{lemma:smallestp}. For Lucas sequences, 
as explained by Lenstra \cite[Lemma 2.5]{Lenstra}, 
	we have 	
		 $q\mid m_p$ if and only if the Frobenius of $p$ belongs to certain conjugacy classes in a field depending only on $q$. But then  one concludes using an effective form of the  Chebotarev density theorem \cite{effche}.
	\end{remark}	
	
	\begin{proof}[Proof of Theorem \ref{thm:lower}]
		Let $S$ be the finite set of primes as in Lemma \ref{lemma:upboundmPi}. If $d\mid \Pi_z$ and $d\nmid \Pi_{z,S}$, then there exists $p\in S$ such that $p\mid d$. So, $m_d=\infty$. Hence, on recalling the definition  \eqref{eq:Mz}, we have
		$M(\Pi_z)=M(\Pi_{z,S})$.
		Lemma \ref{lemma:summd} implies that 
$$
M(\Pi_{z,S})\geq \frac{\varphi(m_{\Pi_{z,S}})}{m_{\Pi_{z,S}}},
$$
since $x_1=1$. 
		But then 
Lemma  \ref{lemma:upboundmPi} yields the desired lower bound. 
	\end{proof}

		\section{Eratosthenes--Legendre sieve and primes in a SDS}\label{sec:E-L}
Let $\{x_{n}\}_{n\in \NN}$ be a SDS satisfying the assumptions of Theorem \ref{thm:gen}. 
	Let $\p$ be  a set of rational primes, and let $z$ and $N$ be two positive parameters. 
We define
$$
	A(N,z)=\left\{n\leq N :\gcd(x_n,\Pi)=1\right\},
$$
where
\begin{equation}\label{eq:Pi}
\Pi=\Pi(z,\p)=\prod_{\substack{p\in \p\\p\leq z}}p.
\end{equation}
We can now prove the following result.

	\begin{lemma}\label{lemma:EL}
Let   $M(\Pi)$ be given by \eqref{eq:Mz}. Then 
		$$
		\#A(N,z)=NM(\Pi)+O(2^{\omega(\Pi)}).
		$$
		\end{lemma}
	
	\begin{proof}
		It follows from inclusion--exclusion that 
		$$
		\#A(N,z)=\sum_{d\mid \Pi} \mu(d)\#A_d ,
		$$
		where $	A_d=\{n\leq N :d \mid x_n\}$.
On appealing to \eqref{eq:LOD}, we deduce that 
		$$
		\#A(N,z)=N\sum_{d\mid \Pi} \frac{\mu(d)}{m_d} +O(2^{\omega(\Pi)}),
		$$
as required. 
	\end{proof}
	
\begin{proof}[Proof of  Theorem \ref{thm:gen}]
In this result we take $\mathcal{P}$ to be the set of all rational primes. 
We have  $\omega(\Pi)\leq z/\log z$ in Lemma \ref{lemma:EL}, by the prime number theorem.
Taking $z= \frac{1}{2}(\log N)(\log\log N)$, we note that $2^{z/\log z}\leq \sqrt{N}$ for sufficiently large values of $N$. Hence 
		it follows from Lemma~\ref{lemma:EL} that
$$
		\#A(N,z)-NM(\Pi)\ll \sqrt{N}.
$$
 Theorem \ref{thm:lower} now yields the statement of the theorem.
  \end{proof}
  
  We next proceed by searching for an upper bound for $\#A(N,z)$.
  
   \begin{definition}
  	Let $\mathcal{F}$ be the family of  strictly increasing continuous functions 
	$$
	f:\R_{\geq 1}\to \R$$ 
	such that $\lim_{z\to \infty} f(z)=\infty$.
  \end{definition}
  
  Bearing this definition in mind, we may establish the following result. 
  
  \begin{lemma}\label{lemma:deff}
  	Let $\{x_n\}_{n\in \NN}$ be a SDS such that, for all but finitely many $n\in \NN$, $x_n$ has a primitive prime divisor. Then there exists $f\in\mathcal{F}$ such that
$$
  		n\leq f(z) \implies x_{n}\leq z ,
$$
	for all $z\geq 1$.
  \end{lemma}
  \begin{proof}
  	Firstly, we notice that 
	 ${x_n}$ cannot admit a constant subsequence, since 
there is a primitive prime divisor for all but finitely many terms. 
Hence	 $\lim_{n\to\infty}x_n=\infty$ and we may form an 
unbounded increasing sequence $y_n=\max_{i\leq n} \{x_i\}$. Let $f_1:\R_{\geq 1}\to \R_{\geq 1}$ be such that $f_1(z)=\max\{n\geq 1: z\geq y_n\}$. Hence, $f_1$ is  an increasing step-function that goes to infinity. Fix $z\geq 1$ and let $n\leq f_1(z)=m$. But then $y_m\leq z$ and $x_n\leq y_m\leq z$. Let $f$ be a strictly increasing continuous function that goes to infinity such that $f(z)\leq f_1(z)$ for all $z\geq 1$. Then, if $n\leq f(z)$, we will have $n\leq f_1(z)$ and so  $x_{n}\leq z $.
  \end{proof}
  \begin{proof}[Proof of Theorem \ref{thm:upperf}]
  	Let $C$ be such that, for all $n>C$, $x_n$ has a primitive prime divisor.
  	Let $n\geq 1$ be divisible by a prime $p$ with $C<p<f(z)$. Let $q$ be a primitive prime divisor of $x_p$. Thus
  	$$
  	q\mid  x_p\leq z
  	$$
  	since $p\leq f(z)$. Hence, $q\mid x_p\mid x_n$ and then $n\notin A(N,z)$.
  	It therefore follows that
  	$$
  	A(N,z)\subseteq\{n\leq N : p\mid n\implies p\leq C \text{ or } p\geq f(z)\}.
  	$$
  	But then an application of the Selberg sieve \cite[Theorem 7.1.1 and Corollary 7.1.2]{cojomurty} yields
  	$$
  	\#\{n\leq N : p\mid n\implies p\leq C \text{ or } p\geq f(z)\}\ll \frac{N}{\log f(z)}+f(z)^2,
  	$$
	which gives our desired upper bound for $\#A(N,z)$.
  \end{proof}
  
    \begin{corollary}\label{cor:upperbigprime}
  	Let $\{x_n\}_{n\geq 1}$ be a SDS such that, for all but finitely many terms, $x_n$ has a primitive prime divisor and let $f$ be as in Lemma \ref{lemma:deff}. Let $g:\R_{\geq f(1)}\to \R$ be such that $f\circ g=\operatorname{Id}$ (which exists since $f$ is continuous and strictly increasing).
  	Then
  	$$
  	\#\left\{n\leq N : p\mid x_n\implies p>g\left(N^{1/3}\right)\right\}\ll \frac{N}{\log N}.
  	$$
  \end{corollary}
  \begin{proof}
  	Fix $N$ and let $z=g(N^{1/3})$. Then,
  	$$
  	\log(f(z))f(z)^2\leq f(z)^{3}=f(g(N^{1/3}))^{3}=(N^{1/3})^{3}= N.
  	$$
	It follows that 
  	$$
  	f(z)^2\leq \frac{N}{\log f(z)},
  	$$
whence
  	$$
  	\#\left\{n\leq N: p\mid x_n\implies p>z\right\}\ll \frac{N}{\log f(z)}+f(z)^2\ll  \frac{N}{\log f(z)}\ll \frac{N}{\log N},
  	$$
by Theorem \ref{thm:upperf}.
  \end{proof}
  
  In Corollary \ref{cor:fib} we applied 
  Theorem \ref{thm:upperf} to get an upper bound for $\#A(N,z)$ for the sequence $\{F_n\}_{n\in \NN}$ of Fibonacci numbers, with 
   $z=\tfrac{1}{2}(\log N)(\log\log N)$. 
In doing so we used the observation that 
$F_n\leq z$ if $n\leq f(z)$, 
with the choice $f(z)=\log_\phi(\sqrt{5}z-1)$.
Note that $f\circ g=\operatorname{Id}$, 
with  $g(z)=\frac{1}{\sqrt{5}}(\phi^z+1)$. But then  Corollary \ref{cor:upperbigprime} implies that
$$
  	\#\left\{n\leq N :
  	p\mid F_n \implies  p> \tfrac{1}{\sqrt{5}}\phi^{N^{1/3}}\right\}\ll \frac{N}{\log N}.
$$
It is interesting to compare this result with  Lemma \ref{lemma:EL}. 
For $z$ of order $\phi^{N^{1/3}}$, the  error term 
$O(2^{\omega(\Pi_z)})$ is 
much larger than $N$, rendering the estimate  meaningless. 

\begin{remark}\label{rem:rat}
We can clearly prove an analogue of Corollary \ref{cor:fib} for the sequence 
$x_n=2^n-1$ of Mersenne numbers, using instead the function $f(z)=\log_2(z+1)$.
In fact, we claim that 
$$
M(\Pi_z)\ll \frac{1}{\log\log z},
$$
for this sequence, where $\Pi_z=\prod_{p\leq z}p$.
To see this we  combine  
Lemma \ref{lemma:EL}
with Theorem~\ref{thm:upperf}, in order to deduce that 
\begin{align*}
NM(\Pi_z)
&\ll  \#A(N,z)+2^{z/\log z}\\
&\ll  \frac{N}{\log f(z)} +f(z)^2+2^{z/\log z}\\
&\ll  \frac{N}{\log \log z} +\log^2 z+2^{z/\log z}.
\end{align*}
The claim follows on dividing through by $N$ and taking the limit as $N\to \infty$.
\end{remark}	

We conclude this section by proving Theorem \ref{thm:cardprimes},
which gives a mild condition under which the set of primes in a  SDS has density $0$.

\begin{proof}[Proof of Theorem \ref{thm:cardprimes}]
Let $\{x_n\}_{n\in \NN}$ be a SDS. 
Let  $\mathcal{A}$ be the set of indexes $n$ such that $x_n=1$ and assume that $\mathcal{A}$ has density $0$. 
	Assume that $x_n=q$ is a prime and write $n=\prod_{i\leq u} p_i^{a_i}$. For all $i\leq u$, let $n_i=n/p_i^{a_i}$ and notice that $x_{n_i}$ is equal to $1$ or $q$ (since it must be a divisor of $q$). Moreover, $\gcd_{i\leq u}(n_i)=1$. If, for all $i\leq u$ we have $x_{n_i}=q$, then $$x_1=x_{\gcd_{i\leq u}(n_i)}=\gcd_{i\leq u}(x_{{n_i}})=q,$$ 
	which is not true. Thus, there exists $i\leq u$ such that $x_{n_i}=1$ and then $n_i\in \mathcal{A}$. Therefore, $n$ can be written as the product of a prime power and an element in $\mathcal{A}$. Let $\p^\infty$ be the set of prime powers and $\p^\infty \cdot \mathcal{A}=\{n\in \N:n=rx, ~r\in \p^\infty, ~x\in \mathcal{A} \}$. In conclusion, we have shown that
	$$
	\{n\in \N:  x_n \text{ is prime }\}\subseteq \p^\infty \cdot \mathcal{A}.
	$$
	Note that $\mathcal{A}$ is divisor closed, meaning that 
	$d\in \mathcal{A}$ whenever 
		  $n\in \mathcal{A}$ and $d\mid n$. This follows, since $x_d\mid x_n=1$.
	We conclude by 
	applying Theorem \ref{mth}, which  shows that if $\mathcal{A}$ has density $0$ and is divisor closed, then $\p^\infty \cdot \mathcal{A}$ has density $0$.
\end{proof}

\begin{remark}\label{rem:necX}
	Notice that the assumption that $\mathcal{A}$ has density $0$ cannot be removed. Indeed, let $\{x_n\}_{n\in \NN}$ be the sequence defined by 
$$
x_n=\begin{cases}
2 &\text{ if $2\mid n$,}\\
1 &\text{ if $2\nmid n$.}
\end{cases}
$$
This is a SDS and we see that $\mathcal{A}$ has density $1/2$. Moreover,  $x_n$ is prime for a positive density set of indices.
\end{remark}
	
	\section{Elliptic divisibility sequences}\label{sec:eds}

	The goal of this section is to prove some properties of elliptic divisibility sequences, and in particular that they satisfy the hypotheses of Theorem \ref{thm:gen}.
	Let $E$ be a rational elliptic curve defined by a Weierstrass equation with integer coefficients. Let $P\in E(\Q)$ be a non-torsion point. For each $n\in \N$, define
	$$
	nP=(x_n:y_n:z_n)\in \mathbb{P}^2(\Q)
	$$
	with $x_n,y_n,z_n\in \ZZ$ such that 
	$\gcd(x_n,y_n,z_n)=1$. If we add the hypothesis $z_n>0$, then the choice of $(x_n:y_n:z_n)$ is unique.  Since the 
	Weierstrass equation defining the curve has integer coefficients, so it follows that
	$x_n=a_nb_n$  and $z_n=b_n^3$, for $a_n,b_n\in \ZZ$. Indeed, if $p^k\| z_n$ then  $p^k\mid x_n^3$, by the equation defining the curve. But then $\gcd(p,y_n)=1$ and $p^k\|x_n^3$.  Hence, $z_n$ is a cube and  $p^j\| x_n$ if $p^{3j}\|z_n$
	The sequence $\{b_n\}_{n\in \NN}$ is an {\em elliptic divisibility sequence} (EDS). Moreover, 
	possibly after a change of variables, we can henceforth assume that $b_1=1$. 
As explained by Silverman \cite[Section 2]{silverman}, it follows from 
 properties of the formal group of an elliptic curve that the 
 sequence $\{b_n\}_{n\in \NN}$ is a SDS.

		Let $p$ be a prime. Note that $m_p$, as given by Definition \ref{def:md}, is the smallest positive integer such that $m_pP$ is the identity 
element $(0:1:0)\in E(\F_p)$ on reducing modulo $p$.  Furthermore, we note that $p\mid b_n$ if and only if $nP$ reduces to the identity modulo $p$.
The next result shows that an  EDS satisfies the hypothesis  in Theorem \ref{thm:gen}.

	\begin{lemma}\label{lemma:bound}
		If $p$ is a prime of good reduction for $E$, then $m_p\mid \#E(\F_p)$. In particular, $m_p\leq p+1+2\sqrt{p}$. 
		If $p$ is not a prime of good reduction for $E$, then there exists a constant $C_E$ depending only on $E$ such that $m_p\leq C_Ep$.
	\end{lemma}

	\begin{proof}
		Assume that $p$ is a prime of good reduction. Then the order of $P$ in $E(\F_p)$ divides the cardinality of the group $E(\F_p)$. But, by definition, $m_p$ is the order of $P$ in $E(\F_p)$. Thus 
		$$
		m_p\mid  \#E(\F_p)\leq p+2\sqrt p+1,
		$$
		by the Hasse bound. 
		
		Assume that $p$ is not a prime of good reduction. 
Then it follows from	\cite[Corollary~C.15.2.1]{arithmetic} that there exists a constant $C'$, depending only on $E$, such that $mP$ is non-singular in $E(\F_p)$ for at least one $m<C'$. The group of non-singular points in $E(\F_p)$ has order at most $2p+1$. We conclude as before.
	\end{proof}
	
	Building on this, 
	we can prove  general upper and lower  bounds  for $m_d$, for any $d\in \NN$. 
		While not directly used in our work, 
the following result shows that $m_d$ can vary quite widely in size. 

	\begin{lemma}\label{lemma:sizes}
		There exists constants $C_1,C_2>0$, depending only on $E$ and $P$, such that
$$
C_1\sqrt{\log d}\leq m_d\leq C_2^{\omega(d)}d,
$$
for any  square-free $d\in \NN$.
	\end{lemma}

\begin{proof}
We begin by proving the lower bound. 
Let $\ve>0$.
As explained by Silverman \cite[Lemma 8]{silverman},  Siegel's theorem implies that 
		\begin{equation}\label{eq:silver}
		(1-\ve)n^2 \hat{h}(P)+O_{\ve,E}(1)\leq \log b_n\leq (1+\ve)n^2 \hat{h}(P)+O_{\ve,E}(1),
		\end{equation}
		where $\hat h(P)>0$ is the 
		canonical height of the point $P$.
We know $d\mid b_{m_d}$ and so it follows that  $\log d\leq \log b_{m_d}$, whence
		$\log d\leq  (1+\eps)m_d^2\hat{h}(P)+O_{\ve,E}(1)$.
The lower bound easily follows.

Turning to the upper bound, let $d=p_1\dots p_r$ and let 
$\Delta_E$ denote the discriminant of $E$. 
It follows from Lemma \ref{lemma:lcm} that  $m_d\leq m_{p_1}\dots m_{p_r}$. 
Applying Lemma \ref{lemma:bound}, we deduce that 
		$$
		m_p\leq 
		\begin{cases}
		p(1+2p^{-1/2}+p^{-1}) & \text{ if $p\nmid \Delta_E$},\\
		C_Ep & \text{ if $p\mid \Delta_E$},
		\end{cases}
		$$
		for any prime $p$. 
But then it follows that 
		$$
		m_d\leq d\prod_{\substack{p\mid d\\p\nmid \Delta_E}}(1+2p^{-1/2}+p^{-1})\prod_{\substack{p\mid d\\p\mid \Delta_E}} C_E\leq C_2^{\omega(d)}d,
		$$
		for a suitable constant $C_2>0$.
\end{proof}
	
	Note that it follows from \eqref{eq:silver} that the EDS $\{b_n\}_{n\in \NN}$ contains only $O(\sqrt{\log N})$ elements in the interval $[1,N]$.

	\section{A Chebotarev argument for EDS} \label{sec:cheb}
	
		In Lemma \ref{lemma:upboundmPi} we proved a lower bound for $\varphi(m_\Pi)/m_\Pi$
	and in 
Lemma \ref{lemma:lowboundmPi} we proved a matching upper bound, subject to some additional hypotheses on the SDS.	
		 The goal of this section is to check that the hypotheses hold for EDS, 
		 associated to a
		 rational elliptic curve $E$ and a non-torsion point $P\in E(\Q)$.
Throughout this section  we shall need to work under the assumption that GRH holds, and we shall assume that $E$ has CM, with $\End(E)=\kappa$. 
		 
		 Given an ideal $I$ in $\kappa$, let $\kappa_I=\kappa(E[I])$ and $L_I=\kappa_I(P/I)$. Since $E$ is defined over $\Q$, it follows that $\kappa$ has class number one (as proved in  \cite[Theorem II.4.3]{advanced}). By $P/I$ we mean a point $Q$ in $E(\overline{\Q})$ such that $\alpha Q=P$ if  $I=(\alpha)$. Notice that the choice of $Q$ is not unique but the field $L_I$ does not depend on this choice. 
		 Building on work of Gupta and Murty \cite{guptamurty, guptamurty2}, we can establish the following result. 
		 
	\begin{lemma}\label{lemma:ramcond}
		Let $p$ a prime that splits in $\kappa$ and let $q\nmid 6p $ be a prime. Let $\pi_p$ be the Frobenius of $p$ in $\kappa$. Then $q\mid m_p$ if and only if there is a prime $\q_1$ over $q$ and  $k_1\geq 1$ such that $\pi_p$ splits completely in $\kappa_{\q_1^{k_1}}$ and does not split completely in $\kappa_{\q_1^{k_1+1}}$ and $L_{\q_1^{k_1}}$.
	\end{lemma}
	\begin{proof}
		Let $\F_{\pi_p}=\OO_\kappa/\pi_p\OO_\kappa$ where $\OO_\kappa$ is the ring of integers of $\kappa$ and notice that $\F_{\pi_p}=\F_p$.
		Recall that the ring of integers of $\kappa$ is a principal ideal domain, since $E$ is a rational elliptic curve. 
		We start by doing the case when $q$ splits in $\kappa$, so that 
		$q=\q_1\q_2$ in $\kappa$. Recall 
		from \cite[Lemma 3]{guptamurty2} 
		that $\#E(\F_p)=N(\pi_p-1)$.  
Given $j\in \NN$, we claim that $\q_1^j$ divides $(\pi_p-1)$ if and only if $\pi_p$ splits completely in  $\kappa_{\q_1^{j}}$.  If $\q_1^j$ divides $(\pi_p-1)$ then
		$
		\pi_p\equiv 1\bmod \q_1^{j},$
		 and it follows that $\pi_p$ acts trivially on $E[\q_1^{j}]$. Therefore, $\pi_p$ splits completely in $\kappa_{\q_1^{j}}$. Conversely, if $\q_1^{j}$ does not divide $(\pi_p-1)$, then $\pi_p$ does not act trivially on $E[\q_1^{j}]$ and $\pi_p$ does not split completely in $\kappa_{\q_1^{j}}$.

		Let $k_1$ (respectively,  $k_2$) be the largest integer such that $\pi_p\equiv 1\bmod{\q_1^{k_1}}$ (respectively,  modulo ${\q_2^{k_2}}$). Thus $\pi_p$ splits completely in $\kappa_{\q_1^{k_1}}$ (respectively,  $\kappa_{\q_2^{k_2}}$) and does not split completely in $\kappa_{\q_1^{k_1+1}}$  (respectively,  $\kappa_{\q_2^{k_2+1}}$). We have $\pi_p-1=\q_1^{k_1}\q_2^{k_2}I$ for $I$ an ideal coprime with $\q_1$ and $\q_2$. 
		Hence $\#E(\F_p)=q^{k_1+k_2}N(I)$ and so  the $q$-primary part of $E(\F_p)$ has order $q^{k_1+k_2}$. 
		Moreover, the $q$-primary part of $E(\F_{p})$ is $E[\q_1^{k_1}]\times E[\q_2^{k_2}]$, since $\pi_p$ acts trivially on $E[\q_1^{k_1}]$ (respectively,  $E[\q_2^{k_2}]$) and does not act trivially on $E[\q_1^{k_1+1}]$ (respectively, $E[\q_2^{k_2+1}]$). Here, we are using that $E[\q_1^{k_1}]$ and $ E[\q_2^{k_2}]$ have trivial intersection, since the two ideals are coprime.
		
		Let $E(\F_{\pi_p})=E[\q_1^{k_1}]\times E[\q_2^{k_2}]\times G$ with $G$ a group of order coprime with $q$. 
		Let $\overline{P}$ be the reduction modulo $\pi_p$ of $P$ and note that $q\mid m_p$ if and only if $q$ divides the order of $\overline{P}$. Write $\overline{P}=P_1+P_2+P_3$ with $P_1\in E[\q_1^{k_1}]$, $P_2 \in E[\q_2^{k_2}]$, and $P_3\in G$. Note that multiplication by $\q_1^{k_1}$ is an isomorphism on $E[\q_2^{k_2}]$ and $G$. Hence $\overline{P}$ is divisible by $\q_1^{k_1}$ if and only if $P_1$ is divisible by $\q_1^{k_1}$. 
		This happens if and only if $P_1=0$. In the same way, $\overline{P}$ is divisible by $\q_2^{k_2}$ if and only if $P_2=0$. Hence, $q$ does not divide the order of $\overline{P}$ if and only if $P_1=0$ and $P_2=0$. It follows that  $\overline{P}$ has order divisible by $q$ if and only if $\overline{P}/\q_1^{k_1}$ does not have a solution in $E(\F_{\pi_p})$ or $\overline{P}/\q_2^{k_2}$ does not have a solution in $E(\F_{\pi_p})$. Hence $\pi_p$ does not have a first-degree prime factor in $L_{\q_1^{k_1}}$ or in $L_{\q_2^{k_2}}$. For more details, see the proof of  \cite[Lemma 2]{guptamurty}. Since $L_{\q_1^{k_1}}/\kappa$ is Galois, this happens if and only if $\pi_p$ does not split completely. In summary, $q$ divides $m_p$ if and only if there exists a prime $\q_1$ over $q$ and a strictly positive integer $k_1$ such that $\pi_p$ splits completely in  $\kappa_{\q_1^{k_1}}$, does not split completely in $\kappa_{\q_1^{k_1+1}}$, and does not split completely in $L_{\q_1^{k_1}}$.
		
		Assume now that $q$ does not split. 
		In this case there is only one prime $\q_1$ in $\kappa$ over $q$. Let $k_1$ be such that $(\pi_p-1)=\q_1^{k_1}I$, with $I$ and $\q_1$ coprime. Proceeding as above, $E(\F_{\pi_p})=E[\q_1^{k_1}]\times G$ with $G$ of order coprime with $q$. 
		We conclude as before.
	\end{proof}
	
	Armed with this result, we can now use the Chebotarev density theorem to assess the density of rational primes that split in the 
	endomorphism ring of $E$ and satisfy the property $q\mid m_p$, for a fixed prime $q$. The analogous result for the property $q\mid \#E(\F_p)$ has been proved by Cojocaru \cite[Section 2.2]{cojocaru}.
	
	\begin{lemma}\label{lemma:cheb}
		Assume that $E$ has CM with $\End(E)=\kappa$, and  assume GRH. Let $q>3$ be a  prime of good reduction. Then there exists $\delta_q\geq 0$ such that 
		$$
\#\{p\leq x: p \text{ splits in }\kappa, q\mid m_p\}= \delta_q\li(x)+O(x^{1/2}\log^3 x),
		$$
where  the implied constant depends on  $E$ and $P$.
	\end{lemma}
	\begin{proof}
	Let $\pi_q(x)=\#\{p\leq x: p \text{ splits in }\kappa, q\mid m_p\}.$
		Given a prime $\q_1$ in $\kappa$, define
		$$
		A_{\q_1,k_1}=\left\{p\leq x : 
		\begin{array}{l}
		\pi_p \text{ splits completely in } \kappa_{\q_1^{k_1}}\\ 
		\pi_p \text{  does not split completely in }\kappa_{\q_1^{k_1+1}} \text{ or } L_{\q_1^{k_1}}
		\end{array}
		 \right\}.
		$$
It follows from  \cite[Theorem 1.1]{effche} (see also \cite[Lemma 7]{guptamurty2}) and an inclusion-exclusion argument, that
		\begin{equation}
		\begin{split}
		\label{eq:Aq1k1}
			\# A_{\q_1,k_1}=~& \frac{\li(x)}{[\kappa_{\q_1^{k_1}}:\kappa]}-\frac{\li(x)}{[\kappa_{\q_1^{k_1+1}}:\kappa]}-\frac{\li(x)}{[L_{\q_1^{k_1}}:\kappa]}+\frac{\li(x)}{[L_{\q_1^{k_1}}\kappa_{\q_1^{k_1}+1}:\kappa]}\\& +O(x^{1/2}(\log x+k_1\log q)).
		\end{split}\end{equation}
		Thus
		$$
		\# A_{\q_1,k_1}=\delta_{\q_1,k_1}\li(x)+O(x^{1/2}(\log x+k_1\log q)),
		$$
		with
		\begin{equation}\label{eq:delta}
		\delta_{\q_1,k_1}=\frac{1}{[\kappa_{\q_1^{k_1}}:\kappa]}-\frac{1}{[\kappa_{\q_1^{k_1+1}}:\kappa]}-\frac{1}{[L_{\q_1^{k_1}}:\kappa]}+\frac{1}{[L_{\q_1^{k_1}}\kappa_{\q_1^{k_1+1}}:\kappa]}.
	\end{equation}
		In a similar way, if $q=\q_1\q_2$, then
		$$
		\# (A_{\q_1,k_1}\cap A_{\q_2,k_2})=\delta_{\q_1,k_1,\q_2,k_2}\li(x)+O(x^{1/2}(\log x+\max\{k_1,k_2\}\log q)),
		$$
		for a suitable constant $\delta_{\q_1,k_1,\q_2,k_2}$.
		
		Assume that $(q)=\q_1\q_2$ splits in $\kappa$.
		By Lemma \ref{lemma:ramcond}, we have
		$$
		\{p\leq x: p \text{ splits in }\kappa, q\mid m_p\}=\left(\cup_{ k_1\geq 1}A_{\q_1,k_1}\right)\bigcup\left(\cup_{k_2\geq 1}A_{\q_2,k_2}\right).
		$$
		Note that, by definition, $A_{k_1,\q_1}\cap A_{k_1',\q_1}=\emptyset$ for $k_1\neq k_1'$. Hence it follows that
		\begin{align*}
		\pi_q(x)&=\sum_{k_1\geq 1}\#A_{\q_1,k_1}+\sum_{k_2\geq 1}\#A_{\q_2,k_2}-\sum_{k_1,k_2\geq 1}\#(A_{\q_1,k_1}\cap A_{\q_2,k_2}).
		\end{align*}
		If $k_1>\frac{\log 2x}{\log q}$, then $q^{k_1}>2x> \#E(\F_p)$. As we noted during the proof of Lemma~\ref{lemma:ramcond}, we have 
		$q^{k_1}\mid \#E(\F_p)$ if $p\in A_{\q_1,k_1}$. Thus, 
		$A_{\q_1,k_1}$ is empty if  $k_1>\frac{\log 2x}{\log q}$.
		Putting $K=\frac{\log 2x}{\log q}$, we deduce that 
		\begin{align*}
		\pi_q(x)=~&\sum_{k_1\leq K}\li(x)\delta_{\q_1,k_1}+\sum_{k_2\leq K}\li(x)\delta_{\q_2,k_2}
		-\sum_{\substack{k_1\leq K\\k_2\leq K}}\li(x)\delta_{\q_1,k_1,\q_2,k_2}\\&+O(K^2x^{1/2}(\log x+K\log q)).
		\end{align*}
		Putting
		$$
		\delta_q=\sum_{1
			\leq k_1}\delta_{\q_1,k_1}+\sum_{1
			\leq k_2}\delta_{\q_2,k_2}-\sum_{1\leq k_1,k_2}\delta_{\q_1,k_1,\q_2,k_2}, 
		$$
		we obtain
		\begin{align*}
\abs{\pi_q(x)-\delta_q \li(x)}
\ll~& 
x^{1/2}\log^3 x
\\
&+
\li(x)\left(\abs{\sum_{k_1> K}\delta_{\q_1,k_1}+\sum_{k_2>K}\delta_{\q_2,k_2}-
\sum_{\max\{k_1,k_2\}>K}\delta_{\q_1,k_1,\q_2,k_2}}\right).
		\end{align*}
		Notice that
		$$
		\sum_{k_1> K}\delta_{\q_1,k_1}+\sum_{k_2>K}\delta_{\q_2,k_2}-
		\sum_{\max\{k_1,k_2\}>K}\delta_{\q_1,k_1,\q_2,k_2}\geq 0
		$$
		and so 
		$$
		\left(\abs{\sum_{k_1> K}\delta_{\q_1,k_1}+\sum_{k_2>K}\delta_{\q_2,k_2}-
			\sum_{\max\{k_1,k_2\}>K}\delta_{\q_1,k_1,\q_2,k_2}}\right)\leq\sum_{k_1> K}\delta_{\q_1,k_1}+\sum_{k_2>K}\delta_{\q_2,k_2}.
		$$
		We have 
		$[\kappa_{\q_i^{k_i}}:\kappa]=q^{k_i-1}(q-1)$, for $i=1,2$, since $q$ is  a prime of good reduction and coprime to $6$.
		By \eqref{eq:delta}, 
		$$
		\delta_{\q_i,k_i}\leq \frac{4}{[\kappa_{\q_i^{k_i}}:\kappa]} =\frac{4}{q^{k_i-1}(q-1)}.
		$$ 
		By this inequality, one can easily show that $\delta_q$ is well-defined.
	Moreover,
		$$
\li(x)		\sum_{k_1>K}\delta_{\q_1,k_1}+\li(x)\sum_{k_2>K}\delta_{\q_2,k_2}\leq \frac{8\li(x)}{q-1}\sum_{k>\log 2x/\log q}\frac{1}{q^{k-1}}\leq 1
		$$
		and then
		$$
		\abs{\pi_q(x)-\delta_q \li(x)}
		\ll
		x^{1/2}\log^3 x.
		$$
		Finally, $\delta_q\geq 0$ since
		$$
		\delta_q=\lim_{x\to \infty}\frac{\pi_q(x)}{\li(x)}\geq 0.
		$$
This completes the proof of the lemma when $q$ splits in $\kappa$; the case in which $q$ does not split is similar. 
	\end{proof}
	
	We will need some control over the dependence of the leading constant $\delta_q$  on $q$, which we achieve in the following result.
	
	\begin{lemma}\label{lemma:delta}
		There exists $\lambda>0$, depending only on $E$ and $P$, such that 
		$\delta_q>q^{-\lambda}$
				for all  primes  $q>3$
				of good reduction.
	\end{lemma}
	\begin{proof}
		
		We do the case $q$ split, the other case being similar. Notice that $$A_{\q_1,k_1}\subseteq \{p\leq x: p \text{ splits in }\kappa, q\mid m_p\},$$ 
		and so $\delta_q\geq \delta_{\q_1,k_1}$. Thus we just need to prove that $\delta_{\q_1,k_1}\geq q^{-\lambda}$ for some $\q_1\mid q$ and $k_1\geq 1$. By \eqref{eq:delta}, we have
		\begin{align*}
		\delta_{\q_1,k_1}&\geq \frac{1}{[\kappa_{\q_1^{k_1}}:\kappa]}-\frac{1}{[L_{\q_1^{k_1}}:\kappa]}-\frac{1}{[\kappa_{\q_1^{k_1+1}}:\kappa]}\\&=\frac{1}{(q-1)q^{k_1-1}}-\frac{1}{[L_{\q_1^{k_1}}:\kappa_{\q_1^{k_1}}][\kappa_{\q_1^{k_1}}:\kappa]}-\frac{1}{(q-1)q^{k_1}}.
		\end{align*}
		If $L_{\q_1^{k_1}}\neq \kappa_{\q_1^{k_1}}$, then $[L_{\q_1^{k_1}}:\kappa_{\q_1^{k_1}}]\geq 2$ and so
		$$
		\delta_{\q_1,k_1}\geq \frac{1}{(q-1)q^{k_1-1}}\left(1-\frac{1}{2}-\frac{1}{q}\right)\geq \frac{1}{q^{k_1+1}}.
		$$
		It is known that there exists a constant $M$ (that depends only on $E$ and $P$) such that for all prime $q$ and $k\geq M$, we have $L_{q^{k}}\neq \kappa_{q^{k}}$ (see \cite[Lemma 14]{hindry}. For a reference on this problem, see \cite[Section 6]{Lombardo_2022}). Let $k\geq M$ and assume that $L_{\q_i^{k}}= \kappa_{\q_i^{k}}$ for $i=1,2$ and $\q_1\q_2=q$. Then, 
		$$
		\kappa_{\q^{k}}=\kappa_{\q_1^{k}}\kappa_{\q_2^{k}}=L_{\q_1^{k}}L_{\q_2^{k}}=L_{q^{k}},
		$$
		contradiction. Hence, $L_{\q_i^{k}}\neq \kappa_{\q_i^{k}}$ for $i$ equal $1$ or $2$. Therefore, for all $k\geq M$,
		$$
		\delta_{\q_i,k}\geq \frac{1}{q^{k+1}}
		$$
		and the lemma easily follows.
	\end{proof}
	
	We are finally ready to prove that the hypotheses 
in	Lemma \ref{lemma:lowboundmPi} are valid for an  EDS, under suitable assumptions. 
Let $\lambda>0$ be the constant appearing in the previous result.

		\begin{proposition}\label{lemma:smallestp}
				Assume that $E$ has CM and  assume GRH. 
				 Let $q>3$ be a 
				prime of good reduction. If $q$ is large enough, then there exists $p\leq q^{3\lambda}$ such that $q\mid m_p$.
	\end{proposition}
	\begin{proof}
		Let $x=q^{3\lambda}$. Lemma 
		\ref{lemma:delta} implies that 
		$$
		\delta_q\li(x)> \frac{\li(x)}{q^\lambda}= \li(x)x^{-\frac{1}{3}}\gg \frac{x^{\frac 23}}{\log x}.
		$$
		Hence Lemma \ref{lemma:cheb} yields
		$$
		\#\{p\leq x: p \text{ splits in }\kappa, q\mid m_p\}\gg \frac{x^{\frac 23 }}{\log x}.
		$$
		Therefore, if $q$ is large enough, then
		$$
		\#\{p\leq x: p \text{ splits in }\kappa, q\mid m_p\}> 0.
		$$
	\end{proof}
	We finally achieved the goal of this section, that was to show, under suitable assumptions, that elliptic divisibility sequences satisfy the hypothesis of Lemma \ref{lemma:lowboundmPi}.
	\begin{proposition}
		Let $E$ be a rational elliptic curve defined by a Weierstrass equation with integer coefficients. Let $P\in E(\Q)$ be a non-torsion point
		and let $\{b_n\}_{n\in \NN}$ be the associated elliptic divisibility sequence. Assume that $E$ has CM and that GRH holds.
		Let $\Pi_z=\prod_{p< z} p$. Then,
		$$
		\frac{1}{\log z} \ll\frac{\varphi(m_{\Pi_{z}})}{m_{\Pi_{z}}} \ll \frac{1}{\log z}.
		$$
	\end{proposition}
	\begin{proof}
		Follows from Lemma \ref{lemma:upboundmPi}, Lemma \ref{lemma:lowboundmPi}, and Proposition \ref{lemma:smallestp}.
	\end{proof}
	    \section{Sieving on Koblitz primes}\label{sec:koblitz}
	    In this section we keep the focus on EDS and return to our argument in Section \ref{sec:E-L}, involving the Eratosthenes--Legendre sieve. However, rather than working with the set of all primes, we consider the effect of taking $\mathcal{P}$ to be the set of primes $p$ for which
$E(\FF_p)$ has prime order. The infinitude of this set is an open question.

	    \begin{conjecture}[Koblitz conjecture]\label{con:kob}
	    	Let $E$ be a non-CM elliptic curve defined
	    	over $\Q$, with conductor $N_E$,  which is not $\Q$-isogenous to a curve with non-trivial $\Q$-torsion. Then there exists a
	    	positive constant $C_E$, depending on $E$, such that
	    	$$
	    	\#\left\{p\leq x:  \text{$p\nmid N_E$ and $\#E(\F_p)$  is prime} \right\}\sim C_E\frac{x}{\log^2x},
	    	$$
		as $x\to \infty$
	    \end{conjecture}

It follows from this conjecture that 
	    	\begin{equation}\label{eq:kob}
	    	\sum_{
		\substack{
		p\nmid N_E\\
		 \#E(\F_p) \text{ is prime}}}\frac{1}{p}<\infty.
	    	\end{equation}
The convergence of this sum has been established  by Cojocaru \cite[Cor.~8]{cojocaru} 
for non-CM  elliptic curves $E$ over $\QQ$, 
under the assumption that GRH holds.

Let $\{b_n\}_{n\in \NN}$ be an EDS 
associated to a non-CM elliptic curve $E$ and a point $P\in E(\QQ)$ of infinite order. 
Let
$$
	A(N,z)=\left\{n\leq N :\gcd(b_n,\Pi)=1\right\},
$$
where $\Pi=\Pi(z,\mathcal{P})$ is given by \eqref{eq:Pi} 
and 
$$
\mathcal{P}=\{p : \#E(\F_p) \text{ is prime}\}.
$$
If $\mathcal{P}$ is empty, then trivially $\#A(N,z)=N$. Since the goal of this section is to estimate $\#A(N,z)$, we shall assume that $\mathcal{P}$ is not empty (which follows if  we assume the Koblitz conjecture).
We begin by establishing the following result.

	    \begin{lemma}\label{lemma:cardAkob}
	    Assume  that Conjecture \ref{con:kob} or GRH holds.
	    	Then there exists a decreasing function
		$h(z):\R_{\geq 2}\to (0,1]$   with $\lim_{z\to\infty} h(z)=\delta \in(0,1)$, such that 
	    	$$
		\#A(N,z)=Nh(z)+O(2^{\omega(\Pi)}).
		$$
	    \end{lemma}

	    \begin{proof}
	    	    	In view of  Lemma \ref{lemma:EL}, it suffices to study $M(\Pi)$.
		Lemma \ref{lemma:summd} yields
		$$
M(\Pi)=\sum_{\substack{j\mid m_\Pi \\\gcd(b_j,\Pi)=1}}\frac{\varphi(m_\Pi/j)}{m_\Pi}.
		$$	
Let $j\mid m_\Pi$. If $j\neq 1$, then $m_p\mid j$ for at least one $p\mid \Pi$,  since $m_p$ is prime for each $p\in \mathcal{P}$. Thus $p\mid b_j$ and then $\gcd(b_j,\Pi)\neq 1$. Let $\mathcal{P}'\subseteq \mathcal{P}$ be a subset of primes such that 
for each $m\in \{m_p: p\in \mathcal{P}\}$ there exists a unique  $p'\in \mathcal{P}'$ for which 
$m_{p'}=m$. Then
	    	$$
		M(\Pi)=\frac{\varphi(m_\Pi)}{m_\Pi}=\prod_{\substack{p\leq z\\ p\in \mathcal{P}'}}\left(1-\frac{1}{m_p}\right).
	    	$$
	    	Let
	    	$$
	    	h(z)=\prod_{\substack{p\leq z\\ p\in \mathcal{P}'}}\left(1-\frac{1}{m_p}\right).
	    	$$
	    	If the set $\{p\leq z: p\in \mathcal{P}'\}$ is empty, we put $h(z)=1$. Notice that, for $z$ large enough, the set is not empty and then $h(z)<1$.
	    	The function $h$ is clearly decreasing, positive, and satisfies $h(z)\leq 1$ for all $z$. 
	    	We may write
	    	$$
	    	h(z)=\exp\left(\log\left(\prod_{
	\substack{	p\leq z\\ p\in \mathcal{P}'}}\left(1-\frac{1}{m_p}\right)
	\right)\right)=\exp\left(-\sum_{
\substack{p\leq z\\ p\in \mathcal{P}'}}\sum_{n\geq 1}\frac{1}{n m_p^n}\right).
	    	$$
	    	It follows from the Hasse bound  that $m_p\geq p/2$, whence
	    	$$
	    	\sum_{\substack{p\leq z\\ p\in \mathcal{P}'}}\sum_{n\geq 2}\frac{1}{n m_p^n}=O\left(1\right).
	    	$$
		Furthermore, it follows from \eqref{eq:kob} that the sum 
		$$
		\sum_{\substack{p\leq z\\ p\in \mathcal{P}'}}\frac{1}{m_p}$$ 
		has a limit as $z\to \infty$. So, $$\sum_{
			\substack{p\leq z\\ p\in \mathcal{P}'}}\sum_{n\geq 1}\frac{1}{n m_p^n}=O(1).$$ The statement of the lemma follows.
	    \end{proof}

Conditionally under 
 Conjecture \ref{con:kob} or GRH, we are now ready to prove that a positive proportion of elements of the sequence $\{b_n\}_{n\in \NN}$ are devoid of primes small prime factors $p$ for which $\#E(\F_p)$  is prime.

	    \begin{theorem}\label{thm:finalkob}
	    Assume that $E$ is a non-CM elliptic curve. 
	    	    Assume  that Conjecture \ref{con:kob} or GRH holds, and let 
		    $\mathcal{P}=\{p : \#E(\F_p) \text{ is prime}\}$. Then 
there exists  $\delta\in (0,1)$ such that 
	    	$$
	    	\#\{n\leq N:  (\text{$p\mid b_n$ and $
		p\in \mathcal{P}$}) \implies p>\log N\}\sim \delta N,
	    	$$
	     as $N\to \infty$.
	    \end{theorem}

	    \begin{proof}
	    	We take  $z=\log N$, noting that 
	    	$$
	    	\#\{p\in \mathcal{P}: p\leq z\}\ll \frac{\log N}{\log \log N},
	    	$$ 
	    	whence  $2^{\omega(\Pi_z)}\ll N^{1/\log\log N}$.
Given $\ve>0$, it follows from Lemma \ref{lemma:cardAkob} that 
	    	$$
	    	N(\delta-\ve)-2^{\omega(\Pi_z)}<\#A(N,z)<N(\delta+\ve)+2^{\omega(\Pi_z)},
	    	$$
		if  $z$ is sufficiently large in terms of $\ve$.
 The statement of the theorem follows. 
	    \end{proof}

\newpage
\appendix

\section{By Sandro Bettin}\label{app}

\begin{center}
\textsc{On the natural density of the product of a divisor closed set\\ and the set of primes}
\end{center}

\bigskip

Given $\mathscr A, \mathscr B\subseteq\mathbb N$, let $\mathscr A\cdot \mathscr B:=\{ab\mid a\in \mathscr A, b\in \mathscr B\}$. Also, for any $m\in\mathbb Z_{\geq0}$, let $\mathbb P_m:=\{n\in\mathbb N\mid \omega(n)\leq m\}$, where $\omega(n)=\sum_{p|n}1$. Finally, we say that $\mathscr A\subseteq \mathbb N$ is divisor closed if whenever $n\in \mathscr A$ and $d|n$ one has $d\in\mathscr A$. Notice that if $\mathscr A$ is divisor closed then so is $\mathscr A \cdot \mathbb P_m$ for all $m$.

The goal of this appendix is to prove the following result.

\begin{theorem}\label{mth}
Let $m\geq0$. If $\mathscr B$ is a divisor closed set of density $0$, then also $\mathscr B\cdot \mathbb P_m$ has density $0$.
\end{theorem}

\begin{remark}
The theorem does not hold without the assumption that $\mathscr B$ is divisor closed. Indeed, if $q_N$ is a sequence of primes going sufficiently slowly to infinity with $N$, then $\mathscr B=\cup_{N}\{n\in (e^{(N-1)^2}, e^{N^2}]  \mid n\equiv  1\ ({\rm mod\ } q_N)\}$ has density $0$, but $\mathscr B\cdot \mathbb P_1$ has density $1$. This can be proved easily using the fundamental lemma of sieve theory. We leave the details to the interested reader.
\end{remark}

A set $\mathscr B$ is divisor closed if and only if its complement $\mathbb N\smallsetminus \mathscr B$ is a set of multiples, where we remind that, given a (finite or infinite) sequence $\mathscr A\subseteq \mathbb N$, the set of multiples of $\mathscr A$ is $\mathcal M(\mathscr A):=\{da\mid d\in\mathbb N, a\in\mathscr A\}$. 
We say that a sequence $\mathscr A\subseteq\mathbb N$ is a Behrend sequence if $\mathcal M(\mathscr A)$ has density $1$.
 We also let ${\bf t}({\mathcal A})=1-\delta (\mathcal M(\mathcal A))$, with $\delta$ denoting the logarithmic density. We state two well-known results on Behrend sequences (see~\cite[(0.68) and Corollary 0.14]{Hall}).

\begin{lemma}\label{il}
We have that $\mathscr A=\{a_1,a_2,\dots \}\subseteq\mathbb N$ is a Behrend sequence if and only if ${\bf t}(\{a_1,\dots, a_n \})\to 0$ as $n\to\infty$. 
\end{lemma}
\begin{lemma}\label{df}
Let $\mathscr A, \mathscr B\subseteq\mathbb N$. If $\mathscr A\cup \mathscr B$ is a Behrend sequence, then so is at least one of $\mathscr A$ and $\mathscr B$.
\end{lemma}

Ruzsa and Tenenbaum~\cite[Theorem~2]{Ruzsa-Tenenbaum} showed that any Behrend sequence can be split into an infinite disjoint union of Behrend sequences. A modification of their proof gives the following result (see also~\cite[Corollary 0.15]{Hall}).
\begin{proposition}\label{mpr}
Let $\mathscr A$ be a Behrend sequence. Then $\mathscr A$ contains infinitely many disjoint  Behrend sequences $\mathscr A_1,\mathscr A_2,\dots$ such that $(a,b)=1$ for all $a\in \mathscr A_i$, $b\in \mathscr A_j$ with $i\neq j$.
\end{proposition}
\begin{proof}
We start with the observation that if $\mathscr B$ is a Behrend sequence and $S\subseteq\mathbb N$ is a finite set then $\mathscr B[S]:=\{b\in\mathcal B \mid (b,s)=1\ \forall s\in S\}$ is a Behrend sequence. This follows immediately from Lemma~\ref{df} since $\mathcal M(S)$ has density smaller than $1$.

Let $\mathscr A=\{a_1,a_2,\dots\}$ and let $\mathscr B_1:=\{a_1,\dots, a_{n_1}\}$ with $n_1$ such that ${\bf t}({\mathscr B_1})<1/2$, as possible by Lemma~\ref{il}. Also, let $\mathscr A^{(1)}:=\mathscr A[\mathscr B_{1}]=\{a_1',a_2',\dots\}$, which is a Behrend set by the above observation. Next, we let $\mathscr B_2:=\{a'_1,\dots, a'_{n_2}\}$, with $n_2$ such that ${\bf t}({\mathscr B_2})<1/3$, and $\mathscr A^{(2)}:=\mathscr A^{(1)}[\mathscr B_{2}]$. 	Repeating this process we define sequences $\mathscr B_1,\mathscr B_2,\dots$ such that ${\bf t}({\mathscr B_j})<1/(j+1)$ for all $j\geq1$. By construction $(b_i,b_j)=1$ if $b_i\in\mathscr B_i ,b_j\in\mathscr B_j $ with $i\neq j$. In order to conclude it suffices to let $\mathscr A_i=\bigcup_{n\in \mathcal N_i}\mathscr {B}_n$, where $\mathcal N_1,\mathcal N_2,\dots$ is any sequence of disjoint infinite subsets of $\mathbb N$. Indeed, ${\bf t}({\mathscr A_i})\leq \inf_{n\in\mathcal N_i}(1/(1+n))=0$ and thus $\mathscr A_i$ is a Behrend sequence by Lemma~\ref{il}.
\end{proof}
\begin{remark}
In general it is not possible to choose the sequences $\mathscr A_i$ so that $\bigcup_i \mathscr A_i=\mathscr A$. For example, this is clearly not possible in the case  $$
\mathscr A=\{p_1p_2\mid p_1,p_2\text{ distinct primes}\}.$$
\end{remark}

Given a sequence $\mathscr A\subseteq\mathbb N$, we let
\begin{align*}
\mathscr A^*:=\{[a_1,\dots,a_m]\,\mid m\geq2,\ a_i\in \mathscr{A},\  (a_1,\dots,a_m)=1\}.
\end{align*}

From Proposition~\ref{mpr} one immediately obtains the following corollary, which would hold also if we fix $m$ to be any integer greater than or equal to $2$.

\begin{corollary}\label{mcll}
If $\mathscr A\subseteq\mathbb N$ is a Behrend sequence, then so is $\mathscr A^*$.
\end{corollary}

\begin{lemma}\label{plm}
We have that $\mathcal M(\mathscr A^*)=(\mathcal M(\mathscr A))^*$. In particular, if $\mathscr A$ is the set of multiples of some set, then so is $\mathscr A^*$.
\end{lemma}
\begin{proof}
Let $n\in \mathcal M(\mathscr A^*)$. Thus $n=\ell \cdot [a_1,\dots,a_m]$ with $\ell\in\mathbb N$, $m\geq2$, $a_i\in  {\mathscr A}$ and $(a_1,\dots,a_m)=1$. For any prime $p$, let $j_p\in\{1,\dots,m\}$ be such that $v_p(a_{j_p})=v_p([a_1,\dots,a_m])$, with $v_p$ denoting the $p$-adic valuation, and let
 $a_{j}':=a_j \prod_{p|\ell, j_p=j}p^{v_p(\ell)}\in\mathcal M(\mathscr A)$ for $j=1,\dots,m$. Then, $n= [a'_1,\dots,a'_m]$ and $(a'_1,\dots,a'_m)=1$ and so $n\in(\mathcal M(\mathscr A))^*$.

Vice versa, let $n\in (\mathcal M(\mathscr A))^*$. Then $n= [\ell_1a_1,\dots,\ell_m a_m]$ with $\ell\in\mathbb N$, $m\geq2$, $a_i\in  {\mathcal A}$ and $(\ell_1a_1,\dots,\ell_ma_m)=1$. Thus $(a_1,\dots,a_m)=1$ and $[a_1,\dots, a_m]$ divides $n$ and so $n\in \mathcal M(\mathscr A^*)$.
\end{proof}

\begin{proposition} \label{mpl}
Let $\mathscr A\subseteq\mathbb N$. Then
\begin{align*}
\mathbb N\smallsetminus\mathcal M (\mathscr A^*)= (\mathbb N\smallsetminus\mathcal M (\mathscr A))\cdot \mathbb P_1.
\end{align*}
\end{proposition}
\begin{proof}
We can assume $1\notin\mathscr A$ since otherwise $\mathcal M (\mathscr A)=\mathbb N$ and the claim is trivial.

Let $n\in (\mathbb N\smallsetminus\mathcal M (\mathscr A))\cdot \mathbb P_1$ so that there exist a prime $p$ and $r\geq0$ such that $p^r|n$ and $n/p^r\notin \mathcal M (\mathscr A)$. If we had $n\in \mathcal M(\mathscr A^*)$, then by Lemma~\ref{plm} we would have $n=[a_1,\dots,a_m]$ for some $m\geq2$ and $a_i\in \mathcal M(\mathscr A)$  with $(a_1,\dots,a_m)=1$.
Since $ (a_1,\dots,a_m)=1$, then there exists $j\in\{1,\dots,m\}$ such that $p\nmid a_j$. In particular, $a_j$ divides $n/p^r$ and so $n/p^r\in \mathcal M(\mathscr A)$, which is a contradiction. Thus, $n\in \mathbb N\smallsetminus\mathcal M (\mathscr A^*)$.

Now, let $n\notin (\mathbb N\smallsetminus\mathcal M (\mathscr A))\cdot \mathbb P_1$. Since $1\notin \mathscr A$ then $(\mathbb N\smallsetminus \mathcal M (\mathscr A))\cdot \mathbb P_1$ contains the prime powers. In particular, the prime factor decomposition of $n$, $n=p_1^{r_1}\cdots p_m^{r_m}$, has
 $m\geq2$.
 By hypothesis for all $j=1,\dots,m$ we have $n/p_j^{r_j}\in \mathcal M(\mathscr A)$. Also, clearly $n=[n/p_1^{r_1},\dots, n/p_m^{r_m}]$ with $(n/p_1^{r_1},\dots, n/p_m^{r_m})=1$ and thus $n\in (\mathcal M(\mathscr A))^*=\mathcal M(\mathscr A^*)$. 
\end{proof}

\begin{proof}[Proof of Theorem~\ref{mth}]
We have $\mathbb P_m=\mathbb P_{1}\cdot \mathbb P_{m-1}$ for $m\geq1$ and thus, since $\mathscr B\cdot \mathbb P_1$ is also divisor closed if $\mathscr B$ is, then by induction we can assume $m=1$.

Let $\mathscr A:=\mathbb N\smallsetminus \mathscr B$. In particular, $\mathscr A$ is a set of multiples of density $1$ and in fact $\mathcal M(\mathscr A)=\mathscr A$, so that $\mathscr A$ is a Behrend sequence. The theorem then follows by Corollary~\ref{mcll} and Proposition~\ref{mpl}.
\end{proof}

\begin{flushright}
\scriptsize{
S. Bettin}\\
\scriptsize{DIMA - Dipartimento di Matematica}\\
\scriptsize{Via Dodecaneso, 35}\\
\scriptsize{16146 Genova}\\
\scriptsize{Italy}\\
\texttt{bettin@dima.unige.it}
\end{flushright}

\end{document}